\documentclass[12pt]{article}

\usepackage[dvipdfm, bookmarksopen=true]{hyperref}
\usepackage[centertags]{amsmath}
\usepackage{amsfonts}
\usepackage{amssymb}
\usepackage{amsthm}
\usepackage{multirow}
\usepackage{array}

\theoremstyle{plain}
\newtheorem{thm}{Theorem}[section]
\newtheorem{cor}[thm]{Corollary}
\newtheorem{lem}[thm]{Lemma}

\newtheorem{prop}[thm]{Proposition}

\theoremstyle{definition}

\newtheorem{defn}[thm]{Definition}

\newtheorem{note}[thm]{Note}
\theoremstyle{remark}

\newcommand{\beast}{\begin{eqnarray*}}
\newcommand{\eeast}{\end{eqnarray*}}

\title{{\bf The classification of Leonard triples\\
 of QRacah type}
}

\author{Hau-wen Huang}

\date{\today}

\begin{document}
\maketitle

\begin{abstract}
Let $\mathbb{K}$ denote an algebraically closed field. Let $V$
denote a vector space over $\mathbb{K}$ with finite positive
dimension. By a Leonard triple on $V$ we mean an ordered triple of
linear transformations in ${\rm End}(V)$ such that for each of
these transformations there exists a basis of $V$ with respect to
which the matrix representing that transformation is diagonal and
the matrices representing the other two transformations are
irreducible tridiagonal. There is a family of Leonard triples said
to have QRacah type. This is the most general type of Leonard
triple. We classify the Leonard triples of QRacah type up to
isomorphism. We show that any Leonard triple of QRacah type
satisfies the $\mathbb{Z}_3$-symmetric Askey-Wilson relations.
\end{abstract}

\section{Leonard pairs and Leonard systems}
We begin by recalling the notion of a Leonard pair. We will use the following terms. Let $X$ denote a square matrix. Then $X$ is called {\it tridiagonal} whenever each nonzero entry lies on either the diagonal, the subdiagonal, or the superdiagonal. Assume $X$ is tridiagonal. Then $X$ is called {\it irreducible} whenever each entry on the subdiagonal is nonzero and each entry on the superdiagonal is nonzero.

\medskip

We now define a Leonard pair. For the rest of this paper $\mathbb{K}$ will denote an algebraically closed field.

\begin{defn}\label{defn1.1}(\cite[Definition 1.1]{ter2001}).
Let $V$ denote a vector space over $\mathbb{K}$ with finite positive dimension. By a {\it Leonard pair} on $V,$ we mean an ordered pair of linear transformations $A:V\rightarrow V$ and $A^*:V\rightarrow V$ that satisfy both (i), (ii) below.
\begin{enumerate}
\item[(i)] There exists a basis for $V$ with respect to which the matrix representing $A$ is irreducible tridiagonal and the matrix representing $A^*$ is diagonal.

\item[(ii)] There exists a basis for $V$ with respect to which the matrix representing $A^*$ is irreducible tridiagonal and the matrix representing $A$ is diagonal.
\end{enumerate}
\end{defn}

\begin{note}\label{note1.1}
According to a common notational convention $A^*$ denotes the conjugate-transpose
of $A.$ We are not using this convention. In a Leonard pair $(A,A^*)$ the linear transformations
$A$ and $A^*$ are arbitrary subject to (i), (ii) above.
\end{note}

For the rest of this paper we fix an integer $d\geq 0.$ Let ${\rm
Mat}_{d+1}(\mathbb{K})$ denote the $\mathbb{K}$-algebra consisting
of all $d+1$ by $d+1$ matrices that have entries in $\mathbb{K}.$
We index the rows and columns by $0,1,\ldots,d.$ We let
$\mathbb{K}^{d+1}$ denote the $\mathbb{K}$-vector space consisting
of all $d+1$ by $1$ matrices that have entries in $\mathbb{K}.$ We
index the rows by $0,1,\ldots,d.$ We view $\mathbb{K}^{d+1}$ as a
left module for ${\rm Mat}_{d+1}(\mathbb{K}).$ For the rest of the
paper let $V$ denote a vector space over $\mathbb{K}$ that has
dimension $d+1.$ Let ${\rm End}(V)$ denote the
$\mathbb{K}$-algebra consisting of all linear transformations from
$V$ to $V.$ Let $\{v_i\}^d_{i=0}$ denote a basis for $V.$ For
$X\in {\rm End}(V)$ and $Y\in {\rm Mat}_{d+1}(\mathbb{K}),$ we say
$Y$ {\it represents $X$ with respect to $\{v_i\}^d_{i=0}$}
whenever $Xv_j=\sum^d_{i=0}Y_{ij}v_i$ for $0\leq j\leq d.$ For
$A\in {\rm End}(V),$ by an {\it eigenvalue} of $A$ we mean a root
of the characteristic polynomial of $A.$ We say that $A$ is {\it
multiplicity-free} whenever it has $d+1$ distinct eigenvalues.
Assume $A$ is multiplicity-free.  Let $\{\theta_i\}^d_{i=0}$
denote an ordering of the eigenvalues of $A.$ For $0\leq i\leq d$
let $V_i$ denote the eigenspace of $A$ associated with $\theta_i.$
Define $E_i\in {\rm End}(V)$ such that $(E_i-I)V_i = 0$ and
$E_iV_j=0$ for $j \not= i$ ($0\leq j \leq d$). Here $I$ denotes
the identity of ${\rm End}(V).$ We call $E_i$ the {\it primitive
idempotent} of $A$ associated with $\theta_i.$



\begin{lem}\label{lem1.0} {\rm (\cite[Lemma~1.3]{ter2001}).}
Let $(A,A^*)$ denote a Leonard pair on $V.$ Then each of $A,A^*$ is multiplicity-free.
\end{lem}

We now define a Leonard system.

\begin{defn}\label{defn1.2} {\rm (\cite[Definition~1.4]{ter2001}).}
By a {\it Leonard system} on $V$ we mean a sequence
$\Phi=(A;\{E_i\}_{i=0}^d;$ $A^*;\{E_i^*\}_{i=0}^d)$ that satisfies
(i)--(v) below.
\begin{enumerate}
\item[(i)] Each of $A,A^*$ is a multiplicity-free element in ${\rm End}(V).$

\item[(ii)] $\{E_i\}_{i=0}^d$ is an ordering of the primitive idempotents of $A.$

\item[(iii)] $\{E_i^*\}_{i=0}^d$ is an ordering of the primitive idempotents of $A^*.$

\item[(iv)]
$
E_iA^*E_j=\left\{
\begin{array}{ll}
0 \qquad &\hbox{if $|i-j|>1,$}\\
\not=0 \qquad &\hbox{if $|i-j|=1$}
\end{array}
\right. \qquad \qquad \hbox{$(0\leq i,j\leq d)$.}
$

\item[(v)]
$
E_i^*AE_j^*=\left\{
\begin{array}{ll}
0 \qquad &\hbox{if $|i-j|>1,$}\\
\not=0 \qquad &\hbox{if $|i-j|=1$}
\end{array}
\right. \qquad \qquad \hbox{$(0\leq i,j\leq d)$.}
$
\end{enumerate}
We refer to $d$ as the {\it diameter} of $\Phi$ and say $\Phi$ is {\it over $\mathbb{K}.$}
\end{defn}

\begin{defn}\label{nota1.1} (\cite[Definition~1.8]{ter2001}).
Let $\Phi=(A;\{E_i\}_{i=0}^d;A^*;\{E_i^*\}_{i=0}^d)$ denote a
Leonard system on $V.$ For $0\leq i\leq d$ let $\theta_i$ (resp.
$\theta_i^*$) denote the eigenvalue of $A$ (resp. $A^*$)
associated with $E_i$ (resp. $E_i^*$). We call
$\{\theta_i\}^d_{i=0}$ (resp. $\{\theta_i^*\}^d_{i=0}$) the {\it
eigenvalue sequence} (resp. {\it dual eigenvalue sequence}) of
$\Phi.$
\end{defn}

\begin{defn}\label{defn1.18} {\rm (\cite[Definition 2.5]{ter2001}).}
Let $(A;\{E_i\}_{i=0}^d;A^*;\{E_i^*\}_{i=0}^d)$ denote a
Leonard system. Define
$$
a_i = {\rm tr}(AE_i^*), \qquad  a_i^* = {\rm tr}(A^*E_i) \qquad
\qquad \hbox{$(0\leq i\leq d),$}
$$
where tr denotes trace.
\end{defn}

The scalars $\{a_i\}^d_{i=0},$ $\{a_i^*\}^d_{i=0}$ have the following interpretation.

\begin{lem}\label{lem1.3}{\rm (\cite[Lemma 10.2]{ter2002}).}
With reference to Definition~\ref{defn1.18},
\begin{align*}
E_i^*AE_i^*&=a_iE_i^*  \qquad \qquad \hbox{$(0\leq i\leq d),$}\\
E_iA^*E_i&=a_i^*E_i  \qquad \qquad \hbox{$(0\leq i\leq d).$}
\end{align*}
\end{lem}

Let $\Phi=(A;\{E_i\}_{i=0}^d;A^*;\{E_i^*\}_{i=0}^d)$ denote a Leonard system on $V.$ Observe that each of the following three sequences is a Leonard system on $V.$
\begin{align*}
\Phi^*&:=(A^*,\{E_i^*\}_{i=0}^d;A;\{E_i\}_{i=0}^d),\\
\Phi^\downarrow&:=(A,\{E_i\}_{i=0}^d;A^*;\{E_{d-i}^*\}_{i=0}^d),\\
\Phi^\Downarrow&:=(A,\{E_{d-i}\}_{i=0}^d;A^*;\{E_i^*\}_{i=0}^d).
\end{align*}
Viewing $*,\downarrow,\Downarrow$ as permutations on the set of all Leonard systems,
\begin{align}
 *^2=\hspace{0.1cm}\downarrow^2 \hspace{0.1cm}  = \hspace{0.1cm} \Downarrow^2 \hspace{0.1cm}  =1, \qquad \qquad \qquad \quad \label{r1.1}\\
\Downarrow*=*\downarrow, \qquad \qquad \downarrow *=*\Downarrow, \qquad \qquad \downarrow\Downarrow \hspace{0.1cm} = \hspace{0.1cm} \Downarrow\downarrow.\label{r1.2}
\end{align}
The group generated by the symbols $*,\downarrow,\Downarrow$ subject to the relations (\ref{r1.1}), (\ref{r1.2}) is the dihedral group $D_4.$ We recall $D_4$ is the group of symmetries of a square, and has $8$ elements. Thus $*,\downarrow,\Downarrow$ induce an action of $D_4$ on the set of all Leonard systems.

\medskip

\begin{defn}\label{defn2.3} (\cite[Section 4]{ter2005}).
Let $\Phi=(A;\{E_i\}_{i=0}^d;A^*;\{E_i^*\}_{i=0}^d)$ denote a
Leonard system on $V.$ Then the pair $(A,A^*)$ forms a Leonard
pair on $V.$ We say this pair is {\it associated} with $\Phi.$
Observe each Leonard system is associated with a unique Leonard
pair.
\end{defn}

\begin{defn}\label{defn2.4} (\cite[Section 4]{ter2005}).
Let $(A,A^*)$ denote a Leonard pair on $V.$ By the {\it associate
class} for $(A,A^*)$ we mean the set of Leonard systems on $V$
which are associated with $(A,A^*).$ Observe this associate class
contains at least one Leonard system $\Phi.$ By \cite[Section
4]{ter2005} this associate class contains $\Phi,$
$\Phi^\downarrow,$ $\Phi^\Downarrow,$
$\Phi^{\downarrow\Downarrow}$ and no other Leonard systems.
\end{defn}



For the rest of this section let $V'$ denote a vector space over $\mathbb{K}$ with dimension $d+1.$ By a {\it $\mathbb{K}$-algebra isomorphism} from ${\rm End}(V)$ to ${\rm End}(V')$ we mean an isomorphism of $\mathbb{K}$-vector spaces $\sigma:{\rm End}(V)\to {\rm End}(V')$ such that $(XY)^\sigma=X^\sigma Y^\sigma$ for all $X,Y\in {\rm End}(V).$

\medskip

It is useful to interpret the concept of isomorphism as follows. Let $\gamma:V\to V'$ denote an isomorphism of $\mathbb{K}$-vector spaces. Define a map $\sigma:{\rm End}(V)\to {\rm End}(V')$ by $X^\sigma=\gamma X\gamma^{-1}$ for all $X\in {\rm End}(V).$ Then $\sigma$ is a $\mathbb{K}$-algebra isomorphism. Conversely let $\sigma:{\rm End}(V)\to {\rm End}(V')$ denote a $\mathbb{K}$-algebra isomorphism. By the Skolem-Noether theorem \cite[Corollary 9.122]{rotman:02} there exists an isomorphism of $\mathbb{K}$-vector spaces $\gamma:V\to V'$ such that $X^\sigma=\gamma X\gamma^{-1}$ for all $X\in {\rm End}(V).$

\medskip

We now recall the notion of isomorphism for Leonard pairs and Leonard systems.

\begin{defn}\label{defn1.16}
Let $(A,A^*)$ denote a Leonard pair on $V.$ Let $(B,B^*)$ denote a Leonard pair on $V'.$ By an {\it isomorphism of Leonard pairs from $(A,A^*)$ to $(B,B^*)$} we mean a $\mathbb{K}$-algebra isomorphism $\sigma:{\rm End}(V)\to {\rm End}(V')$ that sends $A$ to $B$ and $A^*$ to $B^*.$ We say $(A,A^*)$ and $(B,B^*)$ are {\it isomorphic} whenever there exists an isomorphism of Leonard pairs from $(A,A^*)$ to $(B,B^*).$
\end{defn}

Let $\Phi$ denote the Leonard system from Definition~\ref{nota1.1} and let $\sigma:{\rm End}(V)\rightarrow {\rm End}(V')$ denote a $\mathbb{K}$-algebra isomorphism. We write $\Phi^\sigma:=(A^\sigma;\{E_i^\sigma\}^d_{i=0};$ $A^{*\sigma};\{E_i^{*\sigma}\}^d_{i=0})$ and observe $\Phi^\sigma$ is a Leonard system on $V'.$

\begin{defn}\label{defn1.3}
Let $\Phi$ denote a Leonard system on $V.$ Let $\Phi'$ denote a Leonard system on $V'.$ By an {\it isomorphism of Leonard systems from $\Phi$ to $\Phi'$} we mean a $\mathbb{K}$-algebra isomorphism $\sigma:{\rm End}(V)\to {\rm End}(V')$ such that $\Phi^\sigma=\Phi'.$ We say $\Phi,$ $\Phi'$ are {\it isomorphic} whenever there exists an isomorphism of Leonard systems from $\Phi$ to $\Phi'.$
\end{defn}

\begin{defn}\label{defn1.20}
Let $LS=LS(d,\mathbb{K})$ denote the set consisting of the
isomorphism classes of Leonard systems over $\mathbb{K}$ that have
diameter $d.$
\end{defn}

Observe that the $D_4$ action on Leonard systems from above
Definition~\ref{defn2.3} induces a $D_4$ action on the set $LS$
from Definition~\ref{defn1.20}.

\medskip

We recall the notion of an antiautomorphism of ${\rm End}(V).$ By an {\it antiautomorphism of ${\rm End}(V)$} we mean an isomorphism of $\mathbb{K}$-vector spaces $\gamma:{\rm End}(V)\rightarrow {\rm End}(V)$ such that $(XY)^\gamma=Y^\gamma X^\gamma$ for all $X, Y\in {\rm End}(V).$

\begin{lem}\label{thm1.7}
{\rm (\cite[Theorem~6.1]{ter2008}).}
Let $(A,A^*)$ denote a Leonard pair on $V.$ Then there exists a unique antiautomorphism $\dag$ of ${\rm End}(V)$ such that $A^\dag=A$ and $A^{*\dag}=A^*.$ Moreover $X^{\dag\dag}=X$ for all $X\in {\rm End}(V).$
\end{lem}

\begin{defn}\label{defn1.4}
{\rm (\cite[Definition~6.2]{ter2008}).}
Let $(A,A^*)$ denote a Leonard pair on $V.$ By the {\it antiautomorphism which corresponds to} $(A,A^*)$ we mean the map $\dagger$ from Lemma~\ref{thm1.7}. 
\end{defn}


\section{The parameter array of a Leonard system}
Let $\Phi=(A;\{E_i\}^d_{i=0};A^*;\{E_i^*\}^d_{i=0})$ denote a Leonard system on $V.$ In Definition~\ref{nota1.1} we defined the eigenvalue sequence and the dual eigenvalue sequence of $\Phi.$ There are two more parameter sequences of interest to us. In order to define these, we review some results from \cite{ter2001}. For $0\leq i\leq d$ define
\begin{align}\label{e1.5}
U_i=(E_0^*V+E_1^*V+\cdots+E_i^*V)\cap(E_iV+E_{i+1}V+\cdots+E_dV).
\end{align}
By \cite[Lemma~3.8]{ter2001} each of $U_0,U_1,\ldots,U_d$ has dimension one and
\begin{align}\label{e1.6}
V=U_0+U_1+\cdots+U_d \qquad \qquad \hbox{(direct sum).}
\end{align}
The elements $A$ and $A^*$ act on $\{U_i\}^d_{i=0}$ as
follows. By \cite[Lemma~3.9]{ter2001}, both
\begin{align}
(A-\theta_iI)U_i&=U_{i+1}\qquad \hbox{$(0\leq i\leq d-1)$,}\qquad \qquad (A-\theta_dI)U_d=0, \label{e1.1}\\
(A^*-\theta_i^*I)U_i&=U_{i-1}\qquad \hbox{$(1\leq i\leq
d)$,}\qquad \qquad (A^*-\theta_0^*I)U_0=0. \label{e1.2}
\end{align}
Setting $i=0$ in (\ref{e1.5}) we find $U_0=E_0^*V.$ Combining this with (\ref{e1.1}) we find
\begin{align}\label{e1.3}
U_i=(A-\theta_{i-1}I)\cdots (A-\theta_1I)(A-\theta_0I)E_0^*V
\qquad  \qquad \hbox{$(0\leq i\leq d)$.}
\end{align}
Let $v$ denote a nonzero vector in $E_0^*V.$ By (\ref{e1.3}), for $0\leq i\leq d$ the vector $(A-\theta_{i-1}I)\cdots (A-\theta_0I)v$ is a basis for $U_i.$ By this and (\ref{e1.6}) the sequence
\begin{align}\label{e3.13}
(A-\theta_{i-1}I)\cdots (A-\theta_1I)(A-\theta_0I)v \qquad
\hbox{$(0\leq i\leq d)$}
\end{align}
is a basis for $V.$ With respect to this basis the matrices representing $A$ and $A^*$ are
\begin{align}\label{e3.12}
\left(
\begin{array}{cccccc}
\theta_0 & & & & &{\bf 0}\\
1 &\theta_1 & & & &\\
  &1 &\theta_2 & & &\\
  & &\cdot &\cdot & &\\
  & & &\cdot &\cdot &\\
{\bf 0} & & & &1 &\theta_d
\end{array}
\right),\qquad
\left(
\begin{array}{cccccc}
\theta_0^* &\varphi_1 & & & &{\bf 0}\\
  &\theta_1^* &\varphi_2 & & &\\
  & &\theta_2^* &\cdot & &\\
  & & &\cdot &\cdot &\\
  & & & &\cdot &\varphi_d\\
{\bf 0} & & & & &\theta_d^*
\end{array}
\right)
\end{align}
respectively, where $\varphi_1,\varphi_2,\ldots,\varphi_d$ are appropriate scalars in $\mathbb{K}.$ By a {\it $\Phi$-split basis} for $V$ we mean a sequence of the form (\ref{e3.13}), where $v$ is a nonzero vector in $E_0^*V.$ We call $\{\varphi_i\}^d_{i=1}$ the {\it first split sequence} of $\Phi.$ We let $\{\phi_i\}^d_{i=1}$ denote the first split sequence of $\Phi^\Downarrow$ and call this the {\it second split sequence} of $\Phi.$ For notational convenience define $\varphi_0=0,$  $\varphi_{d+1}=0,$ $\phi_0=0,$ $\phi_{d+1}=0.$

\begin{defn}\label{defn1.17} (\cite[Definition~13.3]{ter2005}).
Let $\Phi$ denote a Leonard system on $V.$ Define a map $\natural:{\rm End}(V)\rightarrow {\rm Mat}_{d+1}(\mathbb{K})$ as follows. For all $X\in {\rm End}(V)$ let $X^\natural$ denote the matrix in ${\rm Mat}_{d+1}(\mathbb{K})$ that represents $X$ with respect to a $\Phi$-split basis for $V.$ We observe $\natural:{\rm End}(V)\rightarrow {\rm Mat}_{d+1}(\mathbb{K})$ is a $\mathbb{K}$-algebra isomorphism. We call $\natural$ the {\it natural map} for $\Phi.$
\end{defn}

\begin{defn}\label{defn1.14} (\cite[Definition~22.3]{ter2008}).
Let $\Phi$ denote a Leonard system on $V.$ By the {\it parameter
array} of $\Phi$ we mean the sequence
$(\{\theta_i\}^d_{i=0},\{\theta_i^*\}^d_{i=0},$
$\{\varphi_i\}^d_{i=1},\{\phi_i\}^d_{i=1}),$ where
$\{\theta_i\}^d_{i=0}$ (resp. $\{\theta_i^*\}^d_{i=0}$) is the
eigenvalue sequence (resp. dual eigenvalue sequence) of $\Phi$ and
$\{\varphi_i\}^d_{i=1}$ (resp. $\{\phi_i\}^d_{i=1}$) is the first
split sequence (resp. second split sequence) of $\Phi.$
\end{defn}


\begin{lem}\label{lem2.5} {\rm (\cite[Lemma 5.1]{ter2001}).}
Let $\Phi$ denote a
Leonard system over $\mathbb{K}$ and let
$(\{\theta_i\}^d_{i=0},\{\theta_i^*\}^d_{i=0},$ $\{\varphi_i\}^d_{i=1},\{\phi_i\}^d_{i=1})$
denote the corresponding parameter array. Then the scalars $\{a_i\}^d_{i=0},$
$\{a_i^*\}^d_{i=0}$  from Definition~\ref{defn1.18} are given as follows. If $d=0$ then $a_0=\theta_0$ and $a_0^*=\theta^*_0.$ If $d\geq 1$ then
\begin{align*}
a_0&=\theta_0+\frac{\varphi_1}{\theta_0^*-\theta_1^*},\\
a_i&=\theta_i+\frac{\varphi_i}{\theta_i^*-\theta_{i-1}^*}+\frac{\varphi_{i+1}}{\theta_i^*-\theta_{i+1}^*} \qquad \quad \hbox{$(1\leq i\leq d-1),$}\\
a_d&=\theta_d+\frac{\varphi_d}{\theta_d^*-\theta_{d-1}^*},\\
a_0^*&=\theta_0^*+\frac{\varphi_1}{\theta_0-\theta_1},\\
a_i^*&=\theta_i^*+\frac{\varphi_i}{\theta_i-\theta_{i-1}}+\frac{\varphi_{i+1}}{\theta_i-\theta_{i+1}}
\qquad \quad \hbox{$(1\leq i\leq d-1),$}\\
a_d^*&=\theta_d^*+\frac{\varphi_d}{\theta_d-\theta_{d-1}}.\\
\end{align*}
\end{lem}

\begin{lem}\label{thm1.1} {\rm (\cite[Theorem~1.9]{ter2001}).}
Let
\begin{align}\label{e1.7}
(\{\theta_i\}^d_{i=0}, \{\theta_i^*\}^d_{i=0}, \{\varphi_i\}^d_{i=1}, \{\phi_i\}^d_{i=1})
\end{align}
denote a sequence of scalars taken from $\mathbb{K}.$ Then there
exists a Leonard system $\Phi$ over $\mathbb{K}$ with parameter
array {\rm (\ref{e1.7})} if and only if the following conditions {\rm
(PA1)}--{\rm (PA5)} hold.
\begin{description}
\item[{\rm  (PA1)}] $\theta_i\not=\theta_j,$ \qquad
$\theta_i^*\not=\theta_j^*$ \qquad if $i\not=j$ \qquad \qquad
\hbox{$(0\leq i,j\leq d)$.}

\item[{\rm  (PA2)}] $\varphi_i\not=0,$ \qquad $\phi_i\not=0$
\qquad \qquad \hbox{$(1\leq i\leq d)$.}

\item[{\rm  (PA3)}]
$\varphi_i=\phi_1\displaystyle{\sum\limits^{i-1}_{h=0}\frac{\theta_h-\theta_{d-h}}{
\theta_0-\theta_d}}+(\theta_i^*-\theta_0^*)(\theta_{i-1}-\theta_d)$
\qquad \qquad \hbox{$(1\leq i\leq d)$.}

\item[{\rm  (PA4)}]
$\phi_i=\varphi_1\displaystyle{\sum\limits^{i-1}_{h=0}\frac{\theta_h-\theta_{d-h}}{
\theta_0-\theta_d}}+(\theta_{i}^*-\theta_0^*)(\theta_{d-i+1}-\theta_0)$
\qquad \qquad \hbox{$(1\leq i\leq d)$.}

\item[{\rm  (PA5)}] The expressions
$$
\frac{\theta_{i-2}-\theta_{i+1}}{\theta_{i-1}-\theta_i},\qquad \frac{\theta_{i-2}^*-\theta_{i+1}^*}{\theta_{i-1}^*-\theta_i^*}
$$
~~are equal and independent of $i$ for $2\leq i\leq d-1.$
\end{description}
Moreover, if {\rm (PA1)}--{\rm (PA5)} hold then $\Phi$ is unique
up to isomorphism of Leonard systems.
\end{lem}

\begin{defn}\label{defn1.21} (\cite[Definition~22.1]{ter2008}).
By a {\it parameter array over $\mathbb{K}$ of diameter $d$} we
mean a sequence of scalars $(\{\theta_i\}^d_{i=0},
\{\theta_i^*\}^d_{i=0}, \{\varphi_i\}^d_{i=1},
\{\phi_i\}^d_{i=1})$ taken from $\mathbb{K}$ that satisfies
(PA1)--(PA5).
\end{defn}

\begin{defn}\label{defn1.22}
Let $PA=PA(d,\mathbb{K})$ denote the set consisting of all parameter arrays over $\mathbb{K}$ that have diameter $d.$
\end{defn}

By Lemma~\ref{thm1.1} the map which sends a given Leonard system to its parameter array induces a bijection from $LS$ to $PA.$ Below Definition~\ref{defn1.20} we gave a $D_4$ action on the set $LS.$ This action induces a $D_4$ action on $PA.$ We now describe this action.

\begin{lem}\label{thm1.2}{\rm (\cite[Theorem~1.11]{ter2001}).}
Let $\Phi$ denote a Leonard system with parameter array $(\{\theta_i\}^d_{i=0},$ $\{\theta_i^*\}^d_{i=0},$$\{\varphi_i\}^d_{i=1},$$\{\phi_i\}^d_{i=1}).$ Then {\rm (i)}--{\rm (iii)} hold below.
\begin{enumerate}
\item[{\rm (i)}] The parameter array of $\Phi^*$ is $(\{\theta_i^*\}^d_{i=0},\{\theta_i\}^d_{i=0},\{\varphi_i\}^d_{i=1},\{\phi_{d-i+1}\}^d_{i=1}).$

\item[{\rm (ii)}] The parameter array of $\Phi^\downarrow$ is $(\{\theta_i\}^d_{i=0},\{\theta_{d-i}^*\}^d_{i=0},\{\phi_{d-i+1}\}^d_{i=1},\{\varphi_{d-i+1}\}^d_{i=1}).$

\item[{\rm (iii)}] The parameter array of $\Phi^\Downarrow$ is $(\{\theta_{d-i}\}^d_{i=0},\{\theta_i^*\}^d_{i=0},\{\phi_i\}^d_{i=1},\{\varphi_i\}^d_{i=1}).$
\end{enumerate}
\end{lem}

We mention a result for later use.

\begin{lem}\label{thm1.5} {\rm (\cite[Theorem~17.1]{ter2005}).}
Let $A,A^*$ denote matrices in ${\rm Mat}_{d+1}(\mathbb{K}).$ Assume that $A$ is lower bidiagonal and $A^*$ is upper bidiagonal. Then the following {\rm (i)}, {\rm (ii)} are equivalent.
\begin{enumerate}
\item[{\rm (i)}] The pair $(A,A^*)$ is a Leonard pair on $\mathbb{K}^{d+1}.$

\item[{\rm (ii)}] There exists a parameter array $(\{\theta_i\}^d_{i=0},$$\{\theta_i^*\}^d_{i=0},$$\{\varphi_i\}^d_{i=1},$$\{\phi_i\}^d_{i=1})$ over $\mathbb{K}$ such that
\begin{align*}
A_{ii}=\theta_i, \qquad \qquad A^*_{ii}=\theta_i^* \qquad \qquad \hbox{$(0\leq i\leq d)$,}\\
A_{i,i-1} A^*_{i-1,i}=\varphi_i \qquad \qquad  \qquad
\hbox{$(1\leq i\leq d)$.}
\end{align*}
\end{enumerate}
Suppose {\rm (i)}, {\rm (ii)} hold. For $0\leq i\leq d$ let $E_i$
{\rm (}resp. $E_i^*${\rm )} denote the primitive idempotent of $A$
{\rm (}resp. $A^*${\rm )} associated with $\theta_i$ {\rm (}resp.
$\theta_i^*${\rm )}. Then $(A; \{E_i\}^d_{i=0}; A^*;
\{E_i^*\}^d_{i=0})$ is a Leonard system on $\mathbb{K}^{d+1}$ with
parameter array
$(\{\theta_i\}^d_{i=0},$$\{\theta_i^*\}^d_{i=0},$$\{\varphi_i\}^d_{i=1},$$\{\phi_i\}^d_{i=1}).$
\end{lem}

\section{The Askey-Wilson relations for a Leonard pair}

In this section we recall a few facts about Leonard pairs that will be used later in the paper.

\begin{lem}\label{thm1.6} {\rm (\cite[Theorem~1.5]{tervid}).}
Let $(A,A^*)$ denote a Leonard pair on $V.$ Then there exists a sequence of scalars $\beta,\gamma,\gamma^*,\varrho,\varrho^*,\omega,\eta,\eta^*$ taken from $\mathbb{K}$ such that both
\begin{align}
A^2A^*-\beta AA^*A+A^*A^2-\gamma(AA^*+A^*A)-\varrho \hspace{0.45mm} A^*&=\gamma^*A^2+\omega A+\eta\hspace{0.45mm} I,\label{e19}\\
A^{*2}A-\beta A^*AA^*+AA^{*2}-\gamma^*(A^*A+AA^*)-\varrho^* A&=\gamma A^{*2}+\omega A^*+\eta^* I. \label{e20}
\end{align}
The sequence is uniquely determined by the pair $(A,A^*)$ provided the dimension of
$V$ is at least $4.$
\end{lem}

We refer to (\ref{e19}), (\ref{e20}) as the {\it Askey-Wilson
relations.} Later in the paper we will encounter the Askey-Wilson
relations in another form, said to be $\mathbb{Z}_3$-symmetric.

\begin{lem}\label{lem1.20} {\rm (\cite[Theorem 4.5; Theorem 5.3]{tervid}).}
Let $(A;\{E_i\}^d_{i=0};A^*;\{E_i^*\}^d_{i=0})$ denote a
Leonard system over $\mathbb{K}$ with eigenvalue sequence
$\{\theta_i\}^d_{i=0}$ and dual eigenvalue sequence
$\{\theta_i^*\}^d_{i=0}.$ Let the scalars $a_i,$ $a_i^*$ be as in
Definition~\ref{defn1.18}. Let $\beta,\gamma,\gamma^*,\varrho,\varrho^*,\omega,\eta,\eta^*$
denote a sequence of scalars taken from $\mathbb{K}.$ This sequence satisfies {\rm (\ref{e19})}, {\rm (\ref{e20})} if and only if the following {\rm (i)}--{\rm (ix)} hold.
\begin{enumerate}
\item[{\rm (i)}]
The expressions
$$
\frac{\theta_{i-2}-\theta_{i+1}}{\theta_{i-1}-\theta_i},\qquad \frac{\theta_{i-2}^*-\theta_{i+1}^*}{\theta_{i-1}^*-\theta_i^*}
$$
are both equal to $\beta+1$ for $2\leq i\leq d-1.$

\item[{\rm (ii)}] $\gamma=\theta_{i-1}-\beta\theta_i+\theta_{i+1}$ \qquad \quad $(1\leq i\leq d-1).$

\item[{\rm (iii)}] $\gamma^*=\theta_{i-1}^*-\beta\theta_i^*+\theta_{i+1}^*$ \qquad \quad $(1\leq i\leq d-1).$

\item[{\rm (iv)}] $\varrho=\theta_{i-1}^2-\beta\theta_{i-1}\theta_i+\theta_i^2-\gamma(\theta_{i-1}+\theta_i)$ \qquad \quad $(1\leq i\leq d).$

\item[{\rm (v)}] $\varrho^*=\theta_{i-1}^{*2}-\beta\theta_{i-1}^*\theta_i^*+\theta_i^{*2}-\gamma^*(\theta_{i-1}^*+\theta_i^*)$ \qquad \quad $(1\leq i\leq d).$

\item[{\rm (vi)}] $\omega=a_i^*(\theta_i-\theta_{i+1})+a_{i-1}^*(\theta_{i-1}-\theta_{i-2})-\gamma^*(\theta_i+\theta_{i-1})$ \qquad \quad \hbox{$(1\leq i\leq d).$}

\item[{\rm (vii)}] $\omega=a_i(\theta_i^*-\theta_{i+1}^*)+a_{i-1}(\theta_{i-1}^*-\theta_{i-2}^*)-\gamma(\theta_i^*+\theta_{i-1}^*)$ \qquad \quad \hbox{$(1\leq i\leq d).$}

\item[{\rm (viii)}]
$\eta=a_i^*(\theta_i-\theta_{i-1})(\theta_i-\theta_{i+1})-\gamma^*\theta_i^2-\omega\theta_i$ \qquad \quad \hbox{$(0\leq i\leq d).$}

\item[{\rm (ix)}]
$\eta^*=a_i(\theta_i^*-\theta_{i-1}^*)(\theta_i^*-\theta_{i+1}^*)-\gamma\theta_i^{*2}-\omega\theta_i^*$ \qquad \quad \hbox{$(0\leq i\leq d).$}
\end{enumerate}
In the above lines {\rm (vi)--(ix)}, $\theta_{-1}$ and $\theta_{d+1}$ {\rm (}resp. $\theta^*_{-1}$ and $\theta^*_{d+1}${\rm )} denote scalars in $\mathbb{K}$ that satisfy {\rm (ii)} {\rm (}resp. {\rm (iii)}{\rm )} for $i=0$ and $i=d.$
\end{lem}

\section{Leonard systems of QRacah type; preliminaries}
A bit later in the paper we will consider a family of Leonard systems said to have QRacah type. For these Leonard systems the eigenvalue sequence and dual eigenvalue sequence have a certain form. In this section we consider the form. For the rest of this section let $a,q$ denote nonzero scalars in $\mathbb{K}$ with $q^2\not=\pm 1,$ and let
\begin{align}\label{e1}
\theta_i=aq^{2i-d}+a^{-1}q^{d-2i} \qquad \qquad \hbox{$(0\leq
i\leq d)$.}
\end{align}

We first discuss some necessary and sufficient conditions for
$\{\theta_i\}_{i=0}^d$ to be mutually distinct.

\begin{lem}\label{lem1.2}
We have
$$
\theta_i-\theta_j=(q^{i-j}-q^{j-i})(aq^{i+j-d}-a^{-1}q^{d-i-j}) \qquad \qquad \hbox{$(0\leq i,j\leq d)$}.
$$
\end{lem}
\begin{proof}
Verify this using (\ref{e1}).
\end{proof}

\begin{lem}\label{lem1.6}
The scalars $\{\theta_i\}^d_{i=0}$ are mutually distinct if and only if the following {\rm (i)}, {\rm (ii)} hold.
\begin{enumerate}
\item[{\rm (i)}] $q^{2i}\not=1$ for $1\leq i\leq d.$

\item[{\rm (ii)}] $a^2\not=q^{2d-2i}$ for $1\leq i\leq 2d-1.$
\end{enumerate}
\end{lem}
\begin{proof}
Verify this by using Lemma~\ref{lem1.2}.
\end{proof}

Motivated by Lemma~\ref{lem1.20} we now consider some recursions
satisfied by the sequence (\ref{e1}).

\begin{lem}\label{lem1.5}
Assume $\{\theta_i\}^d_{i=0}$ are mutually distinct. Then
\begin{align*}
\frac{\theta_{i-2}-\theta_{i+1}}{\theta_{i-1}-\theta_i}=q^2+1+q^{-2}\qquad \qquad \hbox{$(2\leq i\leq d-1)$.}
\end{align*}
\end{lem}
\begin{proof}
In the above fraction, evaluate the numerator and denominator
using Lemma~\ref{lem1.2}.
\end{proof}

\begin{lem}\label{lem2.1}
We have
\begin{align*}
\theta_{i-1}-(q^2+q^{-2})\theta_{i}+\theta_{i+1}=0 \qquad \qquad
\hbox{$(1\leq i\leq d-1)$.}
\end{align*}
\end{lem}
\begin{proof}
Verify this using (\ref{e1}).
\end{proof}

\begin{lem}\label{lem2.2}
We have
\begin{align}\label{e1.8}
\theta_{i-1}^2-(q^2+q^{-2})\theta_{i-1}\theta_i+\theta_i^2=-(q^2-q^{-2})^2 \qquad \qquad \hbox{$(1\leq i\leq d).$}
\end{align}
\end{lem}
\begin{proof}
The left-hand side of (\ref{e1.8}) can be factorized into
\begin{align}\label{e1.13}
(\theta_{i-1}-q^2\theta_i)(\theta_{i-1}-q^{-2}\theta_i).
\end{align}
By (\ref{e1}) we find $\theta_{i-1}-q^2\theta_i$ equals $-aq^{2i-d}(q^2-q^{-2})$ and $\theta_{i-1}-q^{-2}\theta_i$ equals $a^{-1}q^{d-2i}(q^2-q^{-2}).$ By these comments (\ref{e1.13}) equals the right-hand side of (\ref{e1.8}).
\end{proof}

In Lemma~\ref{thm1.1} the conditions (PA3), (PA4) involve a
certain sum. We now evaluate this sum for the case (\ref{e1}).

\begin{lem}\label{lem1.4}
We have
\begin{align}\label{e1.12}
\sum^{i-1}_{h=0}\frac{\theta_h-\theta_{d-h}}{\theta_0-\theta_d}=\frac{(q^i-q^{-i})(q^{d-i+1}-q^{i-d-1})}{(q-q^{-1})(q^d-q^{-d})} \qquad \qquad \hbox{$(1\leq i\leq d),$}
\end{align}
provided $\theta_0\not=\theta_d.$
\end{lem}
\begin{proof}
By Lemma~\ref{lem1.2} the summand in the left-hand side of (\ref{e1.12}) equals
$$
\frac{q^{d-2h}-q^{2h-d}}{q^{d}-q^{-d}}.
$$
Therefore the left-hand side of (\ref{e1.12}) involves two geometric series $\{q^{d-2h}\}^{i-1}_{h=0}$ and $\{q^{2h-d}\}^{i-1}_{h=0}.$ We sum the two series to obtain (\ref{e1.12}).
\end{proof}

We finish this section with two miscellaneous results that we will need later.

\begin{lem}\label{lem1.14}
We have
$$
\theta_{d-i}=aq^{d-2i}+a^{-1}q^{2i-d} \qquad \qquad \hbox{$(0\leq i\leq d).$}
$$
\end{lem}
\begin{proof}
Immediate from (\ref{e1}).
\end{proof}

\begin{lem}\label{lem1.7}
Assume $d\geq 1.$ Then
\begin{align}\label{e1.11}
a=\frac{q^{d}\theta_1-q^{d-2}\theta_0}{q^2-q^{-2}}.
\end{align}
\end{lem}
\begin{proof}
From (\ref{e1}) we obtain $\theta_0=aq^{-d}+a^{-1}q^d$ and  $\theta_1=aq^{2-d}+a^{-1}q^{d-2}.$ Solving these equations for $a$ we routinely obtain (\ref{e1.11}).
\end{proof}

\section{Leonard systems of QRacah type}

In this section we define a family of Leonard systems said to have
QRacah type. We discuss some related concepts.

\begin{defn}\label{defn1.15}
Let $\Phi$ denote a
Leonard system on $V,$ as in Definition~\ref{nota1.1}. We say that $\Phi$ has {\it QRacah type}
whenever both (i) $d\geq 3;$ (ii) there exist nonzero $a,b,q\in
\mathbb{K}$ such that $q^2\not=\pm 1$ and
\begin{align}
\theta_i&=aq^{2i-d}+a^{-1}q^{d-2i}\qquad \qquad \hbox{$(0\leq i\leq d)$,}\label{e7}\\
\theta_i^*&=bq^{2i-d}+b^{-1}q^{d-2i} \qquad \qquad \hbox{$(0\leq
i\leq d)$.} \label{e8}
\end{align}
\end{defn}

In view of Definition~\ref{defn1.15}, until further notice we
assume $d\geq 3.$

\begin{defn}\label{defn1.23}
Let $QRAC=QRAC(d,\mathbb{K})$ denote the subset of $LS$ consisting of the isomorphism classes of Leonard systems that have QRacah type.
\end{defn}

Recall the $D_4$ action on the set $LS,$ from below
Definition~\ref{defn1.20}.

\begin{lem}\label{lem2.7}
The set $QRAC$ is closed under the action of $D_4$ on $LS.$
\end{lem}
\begin{proof}
Immediate from Lemma~\ref{thm1.2} and Lemma~\ref{lem1.14}.
\end{proof}

Let $(A,A^*)$ denote a Leonard pair on $V.$ By Definition~\ref{defn2.4} and Lemma~\ref{lem2.7},
if some associated Leonard system has QRacah type then every associated Leonard system has QRacah type; in this case $(A,A^*)$ is said to have {\it QRacah type.}

\section{The parameter arrays of QRacah type}

Let $\Phi$ denote a Leonard system over $\mathbb{K}$ that has QRacah type. In this section we give an explicit form for the parameter array of $\Phi.$

\begin{defn}\label{defn1.9}
Let $(\{\theta_i\}^d_{i=0},\{\theta_i^*\}^d_{i=0},
\{\varphi_i\}^d_{i=1}, \{\phi_i\}^d_{i=1})$ denote a parameter
array over $\mathbb{K}.$ This parameter array is said to have {\it
QRacah type} whenever the corresponding Leonard system has QRacah
type.
\end{defn}

\begin{defn}\label{defn1.24}
Let $PA$-$QRAC=PA$-$QRAC(d,\mathbb{K})$ denote the set consisting
 of the parameter arrays in $PA$ that have QRacah
type.
\end{defn}

Below Definition~\ref{defn1.22} we gave a bijection from $LS$ to
$PA.$ The restriction of that bijection to $QRAC$ forms a
bijection from $QRAC$ to $PA$-$QRAC.$

\begin{lem}\label{lem1.9}
Let
$(\{\theta_i\}^d_{i=0},\{\theta_i^*\}^d_{i=0},\{\varphi_i\}^d_{i=1},\{\phi_i\}^d_{i=1})$
denote a parameter array over $\mathbb{K}$ that has QRacah type.
Let $a,b,q$ denote nonzero scalars in $\mathbb{K}$ such that
$q^2\not=\pm 1$ and {\rm (\ref{e7})}, {\rm (\ref{e8})} hold. Then for all $c\in\mathbb{K}$ the following {\rm (i),} {\rm (ii)} are equivalent.
\begin{enumerate}
\item[{\rm (i)}] $c$ is nonzero and satisfies
\begin{align}
\varphi_i&=a^{-1}b^{-1}q^{d+1}(q^i-q^{-i})(q^{i-d-1}-q^{d-i+1})(q^{-i}-abcq^{i-d-1})(q^{-i}-abc^{-1}q^{i-d-1}),\label{e9}\\
\phi_i&=ab^{-1}q^{d+1}(q^i-q^{-i})(q^{i-d-1}-q^{d-i+1})(q^{-i}-a^{-1}bcq^{i-d-1})(q^{-i}-a^{-1}bc^{-1}q^{i-d-1}) \label{e10}
\end{align}
for $1\leq i\leq d.$

\item[{\rm (ii)}] $c$ is a root of $x^2-\kappa x+1$ where
\begin{align}\label{e1.14}
\kappa=ab^{-1}q^{d-1}+a^{-1}bq^{1-d}+\frac{\phi_1}{(q-q^{-1})(q^d-q^{-d})}.
\end{align}
\end{enumerate}
\end{lem}
\begin{proof}
(i) $\Rightarrow$ (ii): Set $i=1$ in (\ref{e10}) and rearrange terms to obtain
$c+c^{-1}=\kappa.$ Therefore $c$ is a root of $x^2-\kappa x+1.$

(ii) $\Rightarrow$ (i): Note that $c$ is nonzero, and $c^{-1}$ is a root of $x^2-\kappa x+1.$
Therefore $c+c^{-1}=\kappa.$ We substitute this into the left-hand
side of (\ref{e1.14}) and then solve for $\phi_1$ to get
\begin{align}\label{e1.4}
\phi_1=ab^{-1}q^{d+1}(q-q^{-1})(q^{-d}-q^{d})(q^{-1}-a^{-1}bcq^{-d})(q^{-1}-a^{-1}bc^{-1}q^{-d}).
\end{align}
This gives (\ref{e10}) with $i=1.$ To get (\ref{e9}), evaluate the
right-hand side of {\rm (PA3)} using Lemma~\ref{lem1.2},
Lemma~\ref{lem1.4} and (\ref{e1.4}). To get (\ref{e10}) for $2\leq
i\leq d,$ evaluate the right-hand side of {\rm (PA4)} using
Lemma~\ref{lem1.2}, Lemma~\ref{lem1.4} and (\ref{e9}) with $i=1.$
The result follows.
\end{proof}

\begin{cor}\label{lem1.10}
Let
$(\{\theta_i\}^d_{i=0},\{\theta_i^*\}^d_{i=0},\{\varphi_i\}^d_{i=1},\{\phi_i\}^d_{i=1})$
denote a parameter array over $\mathbb{K}$ that has QRacah type.
Let $a,b,q$ denote nonzero scalars in $\mathbb{K}$ such that
$q^2\not=\pm 1$ and {\rm (\ref{e7})}, {\rm (\ref{e8})} hold. Then there exists $c \in \mathbb{K}$ that satisfies the equivalent conditions of Lemma~\ref{lem1.9}. Moreover if $c$ satisfies
these conditions then so does $c^{-1},$ and no other scalar in $\mathbb{K}$ satisfies these conditions.

\end{cor}
\begin{proof}
Immediate from Lemma~\ref{lem1.9}.
\end{proof}

Recall the $D_4$ action on the set $PA,$ from below
Definition~\ref{defn1.22}.

\begin{lem}\label{lem1.16}
The set $PA$-$QRAC$ is closed under the action of $D_4$ on $PA.$
\end{lem}
\begin{proof}
Immediate from Lemma~\ref{lem2.7} and Definition~\ref{defn1.9}.
\end{proof}

\section{A set $QRAC_{red}$}

In (\ref{e7})--(\ref{e10}) we obtained formulae for a parameter
array of QRacah type. Those formulae involve a sequence
of scalars $(a,b,c;q).$ In this section we examine
the properties of this sequence.

\begin{defn}\label{defn1.8}
Let $QRAC_{red}=QRAC_{red}(d,\mathbb{K})$ denote the set of all $4$-tuples $(a,b,c;q)$ of scalars in $\mathbb{K}$ that satisfy the following conditions (RQRAC1)--(RQRAC4).
\begin{description}
\item[{\rm (RQRAC1)}] $a\not=0,$ $b\not=0,$ $c\not=0,$ $q\not=0.$
\item[{\rm (RQRAC2)}] $q^{2i}\not=1$ for $1\leq i\leq d.$
\item[{\rm (RQRAC3)}] Neither of $a^2,$ $b^2$ is among $q^{2d-2},q^{2d-4},\ldots,q^{2-2d}.$
\item[{\rm (RQRAC4)}] None of $abc,$ $a^{-1}bc,$ $ab^{-1}c,$ $abc^{-1}$ is among $q^{d-1},q^{d-3},\ldots,q^{1-d}.$
\end{description}
\end{defn}

\begin{lem}\label{lem1.19}
Let $(\{\theta_i\}^d_{i=0},\{\theta_i^*\}^d_{i=0},\{\varphi_i\}^d_{i=1},\{\phi_i\}^d_{i=1})$
denote a parameter array over $\mathbb{K}$ that has QRacah type.
Let $a,b,c,q$ denote nonzero scalars in $\mathbb{K}$ that satisfy
{\rm (\ref{e7})}--{\rm (\ref{e10})}. Then
$(a,b,c;q)\in QRAC_{red}.$
\end{lem}
\begin{proof}
It is clear that $a,b,c,q$ satisfy (RQRAC1). Conditions (RQRAC2), (RQRAC3) follow from (PA1) and Lemma~\ref{lem1.6}. Condition (RQRAC4) follows from (PA2). Therefore $(a,b,c;q)\in QRAC_{red}.$
\end{proof}

\begin{lem}\label{lem1.15}
Let $(a,b,c;q)\in QRAC_{red}.$ Define
$\{\theta_i\}^d_{i=0},\{\theta_i^*\}^d_{i=0},\{\varphi_i\}^d_{i=1},\{\phi_i\}^d_{i=1}$
by {\rm (\ref{e7})}--{\rm (\ref{e10})}. Then
$(\{\theta_i\}^d_{i=0},\{\theta_i^*\}^d_{i=0},$
$\{\varphi_i\}^d_{i=1},\{\phi_i\}^d_{i=1})$ is a parameter array
over $\mathbb{K}$ that has QRacah type.
\end{lem}
\begin{proof}
We show that $(\{\theta_i\}^d_{i=0},\{\theta_i^*\}^d_{i=0},\{\varphi_i\}^d_{i=1},\{\phi_i\}^d_{i=1})$ is a parameter array over $\mathbb{K}.$ Condition (PA1) follows from Lemma~\ref{lem1.6}, (RQRAC2),  (RQRAC3). Condition (PA2) follows from (RQRAC1), (RQR\\AC2), (RQRAC4). Using Lemma~\ref{lem1.2} and Lemma~\ref{lem1.4} it is routine to verify (PA3), (PA4). Condition (PA5) follows from  Lemma~\ref{lem1.5}. We have shown that $(\{\theta_i\}^d_{i=0},\{\theta_i^*\}^d_{i=0},\{\varphi_i\}^d_{i=1},\{\phi_i\}^d_{i=1})$ is a parameter array over $\mathbb{K}.$ By construction this parameter array has QRacah type.
\end{proof}

\begin{defn}\label{defn1.10}
Let $(a,b,c;q)\in QRAC_{red}.$ Let
$(\{\theta_i\}^d_{i=0},\{\theta_i^*\}^d_{i=0},\{\varphi_i\}^d_{i=1},\{\phi_i\}^d_{i=1})$
denote a parameter array over $\mathbb{K}$ that has QRacah type.
We say these {\it correspond} whenever they satisfy (\ref{e7})--(\ref{e10}).
\end{defn}

Note that each $(a,b,c;q)\in QRAC_{red}$ corresponds to a unique element of $PA$-$QRAC.$

\begin{lem}\label{lem1.8}
Let $(a,b,c;q)\in QRAC_{red}.$ Then all of the following are in $QRAC_{red}.$
\begin{align*}
\begin{split}
\begin{array}{ll}
(a,b,c;q), \qquad  \qquad &((-1)^da,(-1)^db,(-1)^{d+1}c; -q),\\
(a,b,c^{-1};q), \qquad \qquad &((-1)^da, (-1)^db, (-1)^{d+1}c^{-1}; -q),\\
(a^{-1},b^{-1},c^{-1};q^{-1}), \qquad  \qquad&((-1)^da^{-1},(-1)^db^{-1},(-1)^{d+1}c^{-1};-q^{-1}),\\
(a^{-1},b^{-1},c;q^{-1}), \qquad  \qquad &((-1)^da^{-1},(-1)^db^{-1},(-1)^{d+1}c;-q^{-1}).
\end{array}
\end{split}
\end{align*}
Moreover all the above elements correspond to the same element of $PA$-$QRAC.$
\end{lem}
\begin{proof}
This is routinely checked.
\end{proof}

\begin{lem}\label{lem1.11}
Let $(a,b,c;q)\in QRAC_{red}.$ Assume $p\in PA$-$QRAC$ corresponds to $(a,b,c;q).$ Then each element of $QRAC_{red}$ that corresponds to $p$ is listed in Lemma~\ref{lem1.8}.
\end{lem}
\begin{proof}
Suppose we are given $(x,y,z;t)\in QRAC_{red}$ that corresponds to $p.$ By
Lemma~\ref{lem1.5} we find $t^2+t^{-2}=q^2+q^{-2},$ so $t\in\{q,q^{-1},-q,-q^{-1}\}.$ Replacing $q$ by one of
$q,q^{-1},-q,-q^{-1}$ if necessary, we may assume without loss of
generality that $t=q.$ Now $x=a$ by Lemma~\ref{lem1.7} and
similarly $y=b.$ By Corollary~\ref{lem1.10},
$z=c$ or $z=c^{-1}.$ The result follows.
\end{proof}

\begin{defn}\label{defn1.11}
Let $(a,b,c;q)\in QRAC_{red}.$ Let $\Phi$ denote a Leonard system over $\mathbb{K}$ that has QRacah type. We say $(a,b,c;q)$ and $\Phi$ {\it correspond} whenever $(a,b,c;q)$ corresponds to the parameter
array of $\Phi.$
\end{defn}

\begin{cor}\label{lem2.8}
Let $\Phi$ denote a Leonard system over $\mathbb{K}$ that has QRacah type.
Assume $(a,b,c;q)\in QRAC_{red}$ corresponds to $\Phi.$ Then a given element of $QRAC_{red}$ corresponds to $\Phi$ if and only if it is listed in Lemma~\ref{lem1.8}.
\end{cor}
\begin{proof}
Immediate from Lemma~\ref{lem1.8} and Lemma~\ref{lem1.11}.
\end{proof}

\section{A $D_4$ action on $QRAC_{red}$}

Recall the set $QRAC_{red}$ from Definition~\ref{defn1.8}. In this
section we display an action of $D_4$ on $QRAC_{red}.$ We show how this action is related to the
$D_4$ action on $PA$-$QRAC$ given in Lemma~\ref{lem1.16}.

\begin{lem}\label{lem2.3}
There exists a unique $D_4$ action on $QRAC_{red}$ such that
\begin{align}
(a,b,c;q)^*&=(b^{-1},a^{-1},c^{-1};q^{-1}), \label{e4.6}\\
(a,b,c;q)^\downarrow&=(a,b^{-1},c;q), \label{e4.7}\\
(a,b,c;q)^\Downarrow&=(a^{-1},b,c;q) \label{e4.8}
\end{align}
for all $(a,b,c;q)\in QRAC_{red}.$
\end{lem}
\begin{proof}
For all $(a,b,c;q)\in QRAC_{red}$ the sequence on the right
in (\ref{e4.6})--(\ref{e4.8}) is contained in $QRAC_{red}.$ Define
maps $*, \downarrow, \Downarrow$ from $QRAC_{red}$ to $QRAC_{red}$
such that (\ref{e4.6})--(\ref{e4.8}) hold for all $(a,b,c;q)\in QRAC_{red}.$ One checks that these maps satisfy the relations (\ref{r1.1}), (\ref{r1.2}). Therefore the desired $D_4$ action exists. This $D_4$ action is unique since $*, \downarrow, \Downarrow$ generate $D_4.$
\end{proof}

\begin{lem}\label{lem2.4}
For all $g\in D_4$ the following diagram commutes.

\unitlength=1mm
\begin{picture}(120,35)
\multiput(69,6)(0,20){2}{\vector(1, 0){20}}

\put(52,25){{\small $QRAC_{red}$}}
\put(90,25){{\small $PA$-$QRAC$}}
\put(52,5){{\small $QRAC_{red}$}}
\put(90,5){{\small $PA$-$QRAC$}}

\multiput(58,23.5)(40.5,0){2}{\vector(0, -1){15}}
\put(76,28){{\small $cor$}}
\put(76,2){{\small$cor$}}
\put(54.5,16){{\small$g$}}
\put(100.5,16){{\small $g$}}
\end{picture}

\noindent Here ``$cor$'' denotes the correspondence relation from Definition~\ref{defn1.10}.
\end{lem}
\begin{proof}
Without loss we may assume that $g$ is one of $*,\downarrow,
\Downarrow.$ Fix $(a,b,c;q)\in QRAC_{red},$ and let $p$ denote the
corresponding element in $PA$-$QRAC.$ It is routine to
check that $p^g$ and $(a,b,c;q)^g$ correspond according to Definition~\ref{defn1.10}. The result
follows.
\end{proof}


\section{The Askey-Wilson relations for Leonard pairs of QRacah type}

Let $(A,A^*)$ denote a Leonard pair. In Section $3$ we saw that
$A,A^*$ satisfy the Askey-Wilson
relations. In this section we consider what those relations look
like for the case in which $(A,A^*)$ has QRacah type.

\begin{defn}\label{defn1.25}
Let $(a,b,c;q)\in QRAC_{red}.$ Let $(A,A^*)$ denote a Leonard pair
over $\mathbb{K}$ that has diameter $d$ and QRacah type. We say $(a,b,c;q)$ and
$(A,A^*)$ {\it correspond} whenever $(a,b,c;q)$ corresponds to
some Leonard system associated with $(A,A^*).$
\end{defn}

\begin{lem}\label{lem2.9}
Let $(A,A^*)$ denote a Leonard pair on $V$ that has
QRacah type. Assume $(a,b,c;q)\in QRAC_{red}$ corresponds to
$(A,A^*).$ Then the scalars
$\beta,\gamma,\gamma^*,\varrho,\varrho^*,\omega,\eta,\eta^*$ from
Lemma~\ref{thm1.6} are as follows:
\begin{align*}
\beta&=q^2+q^{-2},\qquad \gamma=\gamma^*=0, \qquad \varrho=\varrho^*=-(q^2-q^{-2})^2,\\
\omega&=-(q-q^{-1})^2\big((a+a^{-1})(b+b^{-1})+(c+c^{-1})(q^{d+1}+q^{-d-1})\big), \\  
\eta&=(q-q^{-1})(q^2-q^{-2})\big((c+c^{-1})(a+a^{-1})+(b+b^{-1})(q^{d+1}+q^{-d-1})\big),\\ 
\eta^*&=(q-q^{-1})(q^2-q^{-2})\big((b+b^{-1})(c+c^{-1})+(a+a^{-1})(q^{d+1}+q^{-d-1})\big).
\end{align*}
\end{lem}
\begin{proof}
The scalar $\beta$ is obtained from Lemma~\ref{lem1.20}(i) and Lemma~\ref{lem1.5}. The scalars $\gamma,$ $\gamma^*$ are obtained from Lemma~\ref{lem1.20}(ii),(iii) and Lemma~\ref{lem2.1}. The scalars $\varrho,$ $\varrho^*$ are obtained using Lemma~\ref{lem1.20}(iv),(v) and Lemma~\ref{lem2.2}. To get the scalars $\omega,$ $\eta,$ $\eta^*$ evaluate Lemma~\ref{lem1.20}(vi)--(ix) using Lemma~\ref{lem2.5} and (\ref{e7})--(\ref{e10}).
\end{proof}



\section{The $\mathbb{Z}_3$-symmetric Askey-Wilson relations}

Let $(A,A^*)$ denote a Leonard pair of QRacah type. In the previous section we saw what the corresponding Askey-Wilson relations look like. In this section we show that those Askey-Wilson relations can be put in a form said to be $\mathbb{Z}_3$-symmetric.

\begin{thm}\label{thm2.3}
Let $(A,A^*)$ denote a Leonard pair on $V$ that has QRacah type.
Assume $(a,b,c;q)\in QRAC_{red}$ corresponds to $(A,A^*).$ Then
there exists a unique $A^\varepsilon\in{\rm End}(V)$ such that
\begin{align}
\frac{qA^*A^\varepsilon-q^{-1}A^\varepsilon A^*}{q^2-q^{-2}}+A&=\frac{(b+b^{-1})(c+c^{-1})+(a+a^{-1})(q^{d+1}+q^{-d-1})}{q+q^{-1}}~I,\label{e15}\\
\frac{qA^\varepsilon A-q^{-1}AA^\varepsilon }{q^2-q^{-2}}+A^*&=\frac{(c+c^{-1})(a+a^{-1})+(b+b^{-1})(q^{d+1}+q^{-d-1})}{q+q^{-1}}~I,\label{e16}\\
\frac{qAA^*-q^{-1}A^*A}{q^2-q^{-2}}+A^\varepsilon&=\frac{(a+a^{-1})(b+b^{-1})+(c+c^{-1})(q^{d+1}+q^{-d-1})}{q+q^{-1}}~I.\label{e14}
\end{align}
\end{thm}
\begin{proof}
Define $A^\varepsilon$ such that (\ref{e14}) holds. We show that $A^\varepsilon$ satisfies (\ref{e15}), (\ref{e16}). In these lines eliminate $A^\varepsilon$ using (\ref{e14}), and consider the resulting equations in $A,$ $A^*.$ These equations are the Askey-Wilson relations (\ref{e19}), (\ref{e20}) using the parameters from Lemma~\ref{lem2.9}. These equations hold by Lemma~\ref{lem2.9}. We have shown that there exists $A^\varepsilon$ that satisfies (\ref{e15})--(\ref{e14}). It is clear from (\ref{e14}) that $A^\varepsilon$ is unique.
\end{proof}

We refer to (\ref{e15})--(\ref{e14}) as the {\it $\mathbb{Z}_3$-symmetric Askey-Wilson relations.}


\section{Leonard triples and Leonard triple systems}
Motivated by Theorem~\ref{thm2.3} we now consider the notion of a Leonard triple introduced by Curtin \cite{cur2007}. Until further notice assume $d\geq 0.$

\begin{defn}(\cite[Definition 1.2]{cur2007}).\label{defn1.5}
By a {\it Leonard triple} on $V$ we mean an ordered triple of linear transformations $(A,A^*,A^\varepsilon)$ in ${\rm End}(V)$ such that for each $B\in\{A,A^*,A^\varepsilon\}$ there exists a basis for $V$ with respect to which the matrix representing $B$ is diagonal and the matrices representing the other two linear transformations are irreducible tridiagonal.
\end{defn}

We now define a Leonard triple system.

\begin{defn}\label{defn1.6}
By a {\it Leonard triple system} on $V$ we mean a sequence $\Psi=(A;\{E_i\}_{i=0}^d;A^*;\{E_i^*\}_{i=0}^d;$ $A^\varepsilon;\{E_i^\varepsilon\}_{i=0}^d)$ that satisfies (i)--(vii) below.
\begin{enumerate}
\item[(i)] Each of $A,A^*,A^\varepsilon$ is a multiplicity-free element in ${\rm End}(V).$

\item[(ii)] $\{E_i\}_{i=0}^d$ is an ordering of the primitive idempotents of $A.$

\item[(iii)] $\{E_i^*\}_{i=0}^d$ is an ordering of the primitive idempotents of $A^*.$

\item[(iv)] $\{E_i^\varepsilon\}_{i=0}^d$ is an ordering of the primitive idempotents of $A^\varepsilon.$

\item[(v)] For $B\in\{A^*,A^\varepsilon\},$
$$
E_iBE_j=\left\{
\begin{array}{ll}
0 \qquad &\hbox{if $|i-j|>1,$}\\
\not=0 \qquad &\hbox{if $|i-j|=1$}
\end{array}
\right. \qquad \qquad \hbox{$(0\leq i,j\leq d)$.}
$$

\item[(vi)] For $B\in\{A,A^\varepsilon\},$
$$
E_i^*BE_j^*=\left\{
\begin{array}{ll}
0 \qquad &\hbox{if $|i-j|>1,$}\\
\not=0 \qquad &\hbox{if $|i-j|=1$}
\end{array}
\right. \qquad \qquad \hbox{$(0\leq i,j\leq d)$.}
$$

\item[(vii)] For $B\in\{A,A^*\},$
$$
E_i^\varepsilon BE_j^\varepsilon=\left\{
\begin{array}{ll}
0 \qquad &\hbox{if $|i-j|>1,$}\\
\not=0 \qquad &\hbox{if $|i-j|=1$}
\end{array}
\right. \qquad \qquad \hbox{$(0\leq i,j\leq d)$.}
$$
\end{enumerate}
We refer to $d$ as the {\it diameter} of $\Psi$ and say $\Psi$ is {\it over $\mathbb{K}.$}
\end{defn}

\begin{lem}\label{lem1.1}
Let
$\Psi=(A;\{E_i\}_{i=0}^d;A^*;\{E_i^*\}_{i=0}^d;A^\varepsilon;\{E_i^\varepsilon\}_{i=0}^d)$
denote a sequence of linear transformations in ${\rm End}(V).$ Then $\Psi$ is
a Leonard triple system on $V$ if and only if the following
{\rm (i)}--{\rm (iii)} hold.
\begin{enumerate}
\item[{\rm (i)}]  $(A;\{E_i\}_{i=0}^d;A^*;\{E_i^*\}_{i=0}^d)$ is a Leonard system on $V.$

\item[{\rm (ii)}]  $(A^*;\{E_i^*\}_{i=0}^d;A^\varepsilon;\{E_i^\varepsilon\}_{i=0}^d)$ is a Leonard system on $V.$

\item[{\rm (iii)}]  $(A^\varepsilon;\{E_i^\varepsilon\}_{i=0}^d;A;\{E_i\}_{i=0}^d)$ is a Leonard system on $V.$
\end{enumerate}
\end{lem}
\begin{proof}
Immediate from Definitions~\ref{defn1.2} and \ref{defn1.6}.
\end{proof}

\begin{defn}\label{nota3.1}
Let
$\Psi=(A;\{E_i\}_{i=0}^d;A^*;\{E_i^*\}_{i=0}^d;A^\varepsilon;\{E_i^\varepsilon\}_{i=0}^d)$
denote a Leonard triple system on $V.$ For $0\leq i\leq d$ let
$\theta_i,$ $\theta_i^*,$ $\theta_i^\varepsilon$ denote the
eigenvalues of $A,$ $A^*,$ $A^\varepsilon$ associated with $E_i,$
$E_i^*,$ $E_i^\varepsilon$ respectively. We call
$\{\theta_i\}^d_{i=0},$ $\{\theta_i^*\}^d_{i=0},$
$\{\theta_i^\varepsilon\}^d_{i=0}$ the {\it first, second,
third eigenvalue sequences} of $\Psi.$
\end{defn}

Let $\Psi=(A;\{E_i\}_{i=0}^d;A^*;\{E_i^*\}_{i=0}^d;A^\varepsilon;\{E_i^\varepsilon\}_{i=0}^d)$ denote a Leonard triple system on $V.$ Observe that each of the following five sequences is a Leonard triple system on $V.$
\begin{align*}
\Psi^*&:=(A^*; \{E_i^*\}_{i=0}^d; A; \{E_i\}_{i=0}^d; A^\varepsilon; \{E_i^\varepsilon\}_{i=0}^d), \\
\Psi^\varepsilon&:=(A^\varepsilon; \{E_i^\varepsilon\}_{i=0}^d; A^*; \{E_i^*\}_{i=0}^d; A; \{E_i\}_{i=0}^d),\\
\Psi^\downharpoonright&:=(A; \{E_i\}_{i=0}^d; A^*;
\{E_i^*\}_{i=0}^d; A^\varepsilon;
\{E_{d-i}^\varepsilon\}_{i=0}^d),\\
\Psi^\downarrow&:=(A; \{E_i\}_{i=0}^d; A^*; \{E_{d-i}^*\}_{i=0}^d; A^\varepsilon; \{E_i^\varepsilon\}_{i=0}^d),\\
\Psi^\Downarrow&:=(A; \{E_{d-i}\}_{i=0}^d; A^*; \{E_i^*\}_{i=0}^d;
A^\varepsilon; \{E_i^\varepsilon\}_{i=0}^d).
\end{align*}
Viewing $*,$ $\varepsilon,$ $\downharpoonright,$ $\downarrow,$
$\Downarrow$ as permutations on the set of all Leonard triple
systems,
\begin{align}
*^2=\varepsilon^2=\hspace{0.1cm} \downharpoonright^2\hspace{0.1cm} =\hspace{0.1cm} \downarrow^2\hspace{0.1cm} =\hspace{0.1cm} \Downarrow^2\hspace{0.1cm} =1, \qquad \qquad \qquad \quad~\label{r2.1}\\
*~\varepsilon~*=\varepsilon*\varepsilon, \qquad \downarrow\downharpoonright \hspace{0.1cm} = \hspace{0.1cm} \downharpoonright\downarrow, \qquad \Downarrow \downharpoonright \hspace{0.1cm} = \hspace{0.1cm} \downharpoonright\Downarrow, \qquad \Downarrow\downarrow \hspace{0.1cm} = \hspace{0.1cm} \downarrow\Downarrow, \qquad \label{r2.2}\\
\downharpoonright *=*\downharpoonright,\qquad \qquad \Downarrow *=* \downarrow, \qquad \qquad \downarrow *=* \Downarrow, \qquad \quad \label{r2.3}\\
\qquad \qquad \downarrow\varepsilon=\varepsilon\downarrow, \qquad
\qquad \Downarrow\varepsilon=\varepsilon\downharpoonright, \qquad
\qquad \downharpoonright\varepsilon=\varepsilon\Downarrow. \qquad \quad
\label{r2.4}
\end{align}

The group generated by symbols $*,$ $\varepsilon,$
$\downharpoonright,$ $\downarrow,$ $\Downarrow$ subject to the
relations (\ref{r2.1})--(\ref{r2.4}) is a semidirect product $(\mathbb{Z}_2)^3\rtimes S_3,$ where $\mathbb{Z}_2$ is the cyclic group of order $2$ and $S_3$ is the symmetric group on three letters. The normal subgroup $(\mathbb{Z}_2)^3$ is generated by
$\downharpoonright,$ $\downarrow,$ $\Downarrow$ and the subgroup $S_3$ is generated by $*,$ $\varepsilon.$ By the above comments $*,$ $\varepsilon,$
$\downharpoonright,$ $\downarrow,$ $\Downarrow$ induce an action
of $(\mathbb{Z}_2)^3\rtimes S_3$ on the set of all Leonard triple
systems. We identify $D_4$ with the subgroup of
$(\mathbb{Z}_2)^3\rtimes S_3$ generated by $*,$ $\downarrow,$
$\Downarrow.$

\medskip

Let $\Psi$ denote a Leonard triple system on $V,$ as in Definition~\ref{nota3.1}.
We now display the three eigenvalue sequences of $\Psi^g$ for $g=*,$ $\varepsilon,$
$\downharpoonright,$ $\downarrow,$ $\Downarrow.$

{\center
\begin{tabular}[t]{c|ccc}
\multirow{2}*{$g$} & \multicolumn{3}{c}{\hbox{the eigenvalue sequences of $\Psi^g$}}\\
\cline{2-4}
&{\hbox{1st}} &{\hbox{2nd}} &{\hbox{3rd}}\\
\hline
$*$ &$\{\theta_i^*\}^d_{i=0}$  &$\{\theta_i\}^d_{i=0}$  &$\{\theta_i^\varepsilon\}^d_{i=0}$\\
$\varepsilon$  &$\{\theta_i^\varepsilon\}^d_{i=0}$  &$\{\theta_i^*\}^d_{i=0}$  &$\{\theta_i\}^d_{i=0}$\\
$\downharpoonright$  &$\{\theta_i\}^d_{i=0}$  &$\{\theta_i^*\}^d_{i=0}$  &$\{\theta_{d-i}^\varepsilon\}^d_{i=0}$\\
$\downarrow$  &$\{\theta_i\}^d_{i=0}$  &$\{\theta_{d-i}^*\}^d_{i=0}$  &$\{\theta_i^\varepsilon\}^d_{i=0}$\\
$\Downarrow$  &$\{\theta_{d-i}\}^d_{i=0}$  &$\{\theta_i^*\}^d_{i=0}$  &$\{\theta_i^\varepsilon\}^d_{i=0}$\\
\end{tabular}
\par
}

\medskip

\begin{defn}\label{defn3.1}
Let $\Psi=(A;\{E_i\}^d_{i=0};$
$A^*;\{E_i^*\}^d_{i=0};A^\varepsilon;\{E_i^\varepsilon\}^d_{i=0})$
denote a Leonard triple system on $V.$  Then the triple
$(A,A^*,A^\varepsilon)$ forms a Leonard triple on $V.$ We say this
triple is {\it associated} with $\Psi.$ Observe that each Leonard
triple system is associated with a unique Leonard triple.
\end{defn}

\begin{defn}\label{defn2.5}
Let $(A,A^*,A^\varepsilon)$ denote a Leonard triple on $V.$ By the
{\it associate class} for $(A,A^*,A^\varepsilon)$ we mean the set
of Leonard triple systems on $V$ which are associated with
$(A,A^*,A^\varepsilon).$ Observe that this associate class contains at least one Leonard triple system $\Psi.$ Moreover the associate class is exactly the $(\mathbb{Z}_2)^3$-orbit containing $\Psi.$
\end{defn}

We now define the notion of isomorphism for Leonard triples and Leonard triple systems. For the rest of this section let $V'$ denote a vector space over $\mathbb{K}$ with dimension $d+1.$

\begin{defn}\label{defn2.2}
Let $(A,A^*,A^\varepsilon)$ denote a Leonard triple on $V.$ Let $(B,B^*,B^\varepsilon)$ denote a Leonard triple on $V'.$ By an {\it isomorphism of Leonard triples from $(A,A^*,A^\varepsilon)$ to $(B,B^*,B^\varepsilon)$} we mean a $\mathbb{K}$-algebra isomorphism $\sigma:{\rm End}(V)\to {\rm End}(V')$ that sends $A, A^*,A^\varepsilon$ to $B,B^*,B^\varepsilon$ respectively. We say $(A,A^*,A^\varepsilon)$ and $(B,B^*,B^\varepsilon)$ are {\it isomorphic} whenever there exists an isomorphism of Leonard triples from $(A,A^*,A^\varepsilon)$ to $(B,B^*,B^\varepsilon).$
\end{defn}

Let $\Psi$ denote the Leonard triple system from Definition~\ref{defn1.6}
and let $\sigma: {\rm End}(V)\rightarrow {\rm End}(V')$ denote a
$\mathbb{K}$-algebra isomorphism. We write $\Psi^\sigma:=
(A^\sigma; \{E_i^\sigma \}^d_{i=0};$
$A^{*\sigma};\{E_i^{*\sigma}\}^d_{i=0}; A^{\varepsilon\sigma};
\{E_i^{\varepsilon\sigma}\}^d_{i=0})$ and observe $\Psi^\sigma$ is
a Leonard triple system on $V'.$

\begin{defn}\label{defn1.7}
Let $\Psi$ denote a Leonard triple system on $V.$ Let $\Psi'$ denote a Leonard triple system on $V'.$ By an {\it isomorphism of Leonard triple systems from $\Psi$ to $\Psi'$} we mean a $\mathbb{K}$-algebra isomorphism $\sigma:{\rm End}(V)\to {\rm End}(V')$ such that $\Psi^\sigma=\Psi'.$ We say $\Psi,$ $\Psi'$  are {\it isomorphic} whenever there exists an isomorphism of Leonard triple systems from $\Psi$ to $\Psi'.$
\end{defn}

\begin{defn}\label{defn2.7}
Let $LTS=LTS(d,\mathbb{K})$ denote the set consisting of all
isomorphism classes of Leonard triple systems over $\mathbb{K}$
that have diameter $d.$
\end{defn}

Observe that the $(\mathbb{Z}_2)^3\rtimes S_3$ action on Leonard
triple systems from above Definition~\ref{defn3.1} induces a
$(\mathbb{Z}_2)^3\rtimes S_3$ action on $LTS.$

\section{Leonard triple systems of QRacah type}

In this section we define a family of Leonard triple systems said
to have QRacah type. We discuss some related concepts.

\begin{defn}\label{defn1.12}
Let $\Psi$
denote a Leonard triple system on $V,$ as in Definition~\ref{nota3.1}. We say that $\Psi$ has {\it
QRacah type} whenever both (i) $d\geq 3;$ (ii) there exist nonzero
$a,b,c,q\in \mathbb{K}$ such that $q^2\not=\pm 1$ and
\begin{align}
\theta_i&=aq^{2i-d}+a^{-1}q^{d-2i} \qquad \qquad \hbox{$(0\leq i\leq d),$} \label{e11}\\
\theta_i^*&=bq^{2i-d}+b^{-1}q^{d-2i} \qquad \qquad \hbox{$(0\leq i\leq d),$} \label{e12}\\
\theta_i^\varepsilon&=cq^{2i-d}+c^{-1}q^{d-2i} \qquad \qquad
\hbox{$(0\leq i\leq d).$} \label{e13}
\end{align}
\end{defn}

Until further notice assume $d\geq 3.$

\begin{defn}\label{defn3.3}
Let $T$-$QRAC=T$-$QRAC(d,\mathbb{K})$ denote the subset of $LTS$
consisting of the isomorphism classes of Leonard triple systems
that have QRacah type.
\end{defn}

\begin{lem}\label{lem1.17}
The set $T$-$QRAC$ is closed under the action of $(\mathbb{Z}_2)^3\rtimes S_3$ on $LTS.$
\end{lem}
\begin{proof}
Immediate from Lemma~\ref{lem1.14} and the table above Definition~\ref{defn3.1}.
\end{proof}

Let $(A,A^*,A^\varepsilon)$ denote a Leonard triple on $V.$ By Definition~\ref{defn2.5} and
Lemma~\ref{lem1.17}, if some associated Leonard triple system has QRacah
type then every associated Leonard triple system has QRacah type; in this
case $(A,A^*,A^\varepsilon)$ is said to have {\it QRacah type}.

\medskip

Given Theorem~\ref{thm2.3}, it is natural to ask whether every
Leonard pair of QRacah type extends to a Leonard triple of QRacah
type. The next two sections are devoted to this issue.

\section{The condition for $A^{\varepsilon}$ to be multiplicity-free}

Let $(A,A^*)$ denote a Leonard pair on $V$ that has QRacah type.
Fix $(a,b,c;q)\in QRAC_{red}$ which corresponds to $(A,A^*),$ and
let $A^\varepsilon \in {\rm End}(V)$ be the corresponding element
from Theorem~\ref{thm2.3}. Let
$\Phi=(A;\{E_i\}^d_{i=0};A^*;\{E_i^*\}^d_{i=0})$ denote the
Leonard system that corresponds to $(a,b,c;q).$ Our next goal is
to find necessary and sufficient conditions for the triple
$(A,A^*,A^\varepsilon)$ to be a Leonard triple. To this end we
first determine when $A^\varepsilon$ is multiplicity-free. We
recall some notation. For any $x,t\in \mathbb{K}$ define
\begin{align*}
(x;t)_n:=(1-x)(1-xt)\cdots(1-xt^{n-1}) \qquad \qquad n=0,1,2,\ldots
\end{align*}
and interpret $(x;t)_0:=1.$

\begin{defn}\label{defn3.2}
Define $M$ to be the upper triangular matrix in ${\rm Mat}_{d+1}(\mathbb{K})$ with entries
\begin{align}\label{e4.4}
M_{ij}:=(-1)^ib^{-i}c^{j-i}q^{i^2+(d-2i)j}~
\frac{(q^{2i+2};q^2)_{j-i}(q^{2i-2d};q^2)_{j-i}(a^{-1}b^{-1}c^{-1}q^{d-2j+1};q^2)_{j-i}}{(q^2;q^2)_{j-i}}
\end{align}
for $0\leq i\leq j\leq d.$ Note that the diagonal entries of $M$
are
\begin{align}\label{e4.5}
M_{ii}=(-1)^ib^{-i}q^{i(d-i)} \qquad \qquad \hbox{$(0\leq i\leq
d).$}
\end{align}
These entries are nonzero so $M$ is invertible. For notational convenience define
$M_{ij}=0$ if $i$ or $j$ is among $-1,~d+1.$
\end{defn}

\begin{defn}\label{defn2.1}
Define a map $\rho:{\rm End}(V)\rightarrow {\rm Mat}_{d+1}(\mathbb{K})$ by
\begin{align*}
X^\rho:=M^{-1}X^\natural M\qquad \quad \hbox{for all $X\in {\rm End}(V),$}
\end{align*}
where $\natural$ is the natural map for $\Phi$ from Definition~\ref{defn1.17}. Observe that $\rho$ is a $\mathbb{K}$-algebra isomorphism.
\end{defn}

To evaluate $A^{\varepsilon\rho}$ we need some lemmas.

\begin{lem}\label{lem3.2}
For $0\leq i\leq j\leq d$ with $(i,j)\not=(0,d)$ we have
\begin{align}
M_{i-1,j}&=a^{-1}q^{i+j-d-1}~\frac{(q^{i}-q^{-i})(q^{d-i+1}-q^{i-d-1})(abc-q^{d-2i+1})}{q^{i-j-1}-q^{j-i+1}}~M_{ij}, \label{e3.3}\\
M_{i,j+1}&=a^{-1}b^{-1}q^{j-i}~\frac{(q^{j+1}-q^{-j-1})(q^{d-j}-q^{j-d})(abc-q^{d-2j-1})}{q^{j-i+1}-q^{i-j-1}}~M_{ij}. \label{e3.4}
\end{align}
\end{lem}
\begin{proof}
Use Definition~\ref{defn3.2}.
\end{proof}

\begin{lem}\label{lem3.3}
For $0\leq i< j\leq d$ we have
\begin{align}
M_{i+1,j}&=a\hspace{0.45mm}q^{d-i-j}~\frac{q^{i-j}-q^{j-i}}{(q^{i+1}-q^{-i-1})(q^{d-i}-q^{i-d})(abc-q^{d-2i-1})}~M_{ij},\label{e4.2}\\
M_{i,j-1}&=ab\hspace{0.45mm}q^{i-j+1}~\frac{q^{j-i}-q^{i-j}}{(q^{j}-q^{-j})(q^{d-j+1}-q^{j-d-1})(abc-q^{d-2j+1})}~M_{ij}.\label{e4.3}
\end{align}
\end{lem}
\begin{proof}
To verify (\ref{e4.2}), replace $i$ by $i+1$ in (\ref{e3.3}) and solve the resulting equation for $M_{i+1,j}.$ To verify (\ref{e4.3}), replace $j$ by $j-1$ in (\ref{e3.4}) and solve the resulting equation for $M_{i,j-1}.$
\end{proof}

\begin{lem}\label{lem3.0}
The matrix $A^{\varepsilon\natural}\in {\rm Mat}_{d+1}(\mathbb{K})$ is irreducible tridiagonal with entries
\begin{align}
(A^{\varepsilon\natural})_{i,i-1}&=-b^{-1}q^{d-2i+1}, \label{e3.7}\\
(A^{\varepsilon\natural})_{i-1,i}&=a^{-2}b^{-1}(q^{i}-q^{-i})(q^{d-i+1}-q^{i-d-1})(abc-q^{d-2i+1})(abc^{-1}-q^{d-2i+1}) \label{e3.8}
\end{align}
for $1\leq i\leq d$ and
\begin{align}
(A^{\varepsilon\natural})_{ii}&=a^{-1}b^{-1}q^{d-2i}\big(q^{d+1}+q^{-d-1}-q^{d-2i}(q+q^{-1})\big)+(c+c^{-1})q^{d-2i} \qquad \qquad \label{e3.9}
\end{align}
for $0\leq i\leq d.$
\end{lem}
\begin{proof}
The matrix $A^\natural$ (resp. $A^{*\natural}$) is given on the
left (resp. right) in (\ref{e3.12}). 
After a short computation using (\ref{e14}) we find that the matrix $A^{\varepsilon\natural}$ is tridiagonal with entries
\begin{align}
(A^{\varepsilon\natural})_{i,i-1}&=\frac{q^{-1}\theta_{i}^*-q\theta_{i-1}^*}{q^2-q^{-2}}, \label{e3.15}\\
(A^{\varepsilon\natural})_{i-1,i}&=\varphi_i~\frac{q^{-1}\theta_i-q\theta_{i-1}}{q^2-q^{-2}} \label{e3.16}
\end{align}
for $1\leq i\leq d,$ and
\begin{align}
(A^{\varepsilon\natural})_{ii}&=\frac{(a+a^{-1})(b+b^{-1})+(c+c^{-1})(q^{d+1}+q^{-d-1})}{q+q^{-1}}-\frac{\theta_i\theta_i^*}{q+q^{-1}}-\frac{q\varphi_i-q^{-1}\varphi_{i+1}}{q^2-q^{-2}} \label{e3.17}
\end{align}
for $0\leq i\leq d.$ To obtain (\ref{e3.7})--(\ref{e3.9}),
evaluate (\ref{e3.15})--(\ref{e3.17}) using
(\ref{e7})--(\ref{e9}). By (RQRAC1), the right-hand side of (\ref{e3.7}) is nonzero for $1\leq i\leq d.$ By (RQRAC2) and (RQRAC4) the right-hand side of (\ref{e3.8}) is nonzero for $1\leq i\leq d.$ Therefore the tridiagonal matrix $A^{\varepsilon\natural}$ is
irreducible.
\end{proof}

In Lemma~\ref{lem3.0} we gave the entries of $A^{\varepsilon\natural}$. For notational convenience define $(A^{\varepsilon\natural})_{ij}=0$ if $i$ or $j$ is among $-1,$ $d+1.$

\begin{lem}\label{lem3.4}
For $0\leq i\leq j\leq d$ with $(i,j)\not=(0,d)$ we have
\begin{align} (A^{\varepsilon\natural})_{i,i-1}M_{i-1,j}&=a^{-1}b^{-1}q^{j-i}~\frac{(q^{i}-q^{-i})(q^{d-i+1}-q^{i-d-1})(abc-q^{d-2i+1})}{q^{j-i+1}-q^{i-j-1}}~M_{ij}, \label{e3.5}\\
(A^{\varepsilon\natural})_{i,i+1}M_{i+1,j}&=a^{-1}b^{-1}q^{d-i-j}(q^{i-j}-q^{j-i})(abc^{-1}-q^{d-2i-1})M_{ij}. \label{e3.6}
\end{align}
\end{lem}
\begin{proof}
To obtain (\ref{e3.5}) evaluate the left-hand side of (\ref{e3.5}) using (\ref{e3.3}) and (\ref{e3.7}).  Concerning (\ref{e3.6}), first assume $i=j.$ Then (\ref{e3.6}) holds since each side is zero. Next assume $i<j.$ In this case, (\ref{e3.6}) is verified by evaluating the left-hand side using (\ref{e4.2}) and (\ref{e3.8}).
\end{proof}

\begin{prop}\label{prop3.1}
We have
\begin{align}\label{e3.19}
A^{\varepsilon\rho}=\left(
\begin{array}{cccccc}
\theta^\varepsilon_0 & & & & &{\bf 0}\\
1 &\theta^\varepsilon_1 & & & &\\
  &1 &\theta^\varepsilon_2 & & &\\
  & &\cdot &\cdot & &\\
  & & &\cdot &\cdot &\\
{\bf 0} & & & &1 &\theta^\varepsilon_d
\end{array}
\right),
\end{align}
where $\{\theta_i^\varepsilon\}^d_{i=0}$ are from {\rm
(\ref{e13})}.
\end{prop}
\begin{proof}
Let $B$ denote the matrix on the right in (\ref{e3.19}). To show
that $A^{\varepsilon\rho}=B,$ it suffices to show that
$A^{\varepsilon\natural} M=MB.$ To do this, for $0\leq i,j\leq d$
we show that the $(i,j)$-entry of $A^{\varepsilon\natural} M$
equals the $(i,j)$-entry of $MB.$ In other words, it suffices to
show
\begin{align}\label{e3.2}
(A^{\varepsilon\natural})_{i,i-1}M_{i-1,j}+(A^{\varepsilon\natural})_{ii}M_{ij}+(A^{\varepsilon\natural})_{i,i+1}M_{i+1,j}=\theta^\varepsilon_jM_{ij}+M_{i,j+1}.
\end{align}
To verify (\ref{e3.2}) we consider the following four cases:
$$
\begin{array}{lll}
\hbox{(I) $j-i<-1;$} \qquad \qquad &\hbox{(II) $j-i=-1;$}\\
\hbox{(III) $j-i\geq 0$ and $(i,j)\not=(0,d);$} \qquad \qquad
&\hbox{(IV) $(i,j)=(0,d).$}
\end{array}
$$

Case (I): Each summand in (\ref{e3.2}) is zero so
(\ref{e3.2}) holds.

Case (II): In this case (\ref{e3.2}) reduces to
\begin{align}\label{e3.18}
(A^{\varepsilon\natural})_{i,i-1} M_{i-1,i-1}=M_{ii}.
\end{align}
Using (\ref{e4.5}) we find
$M_{i-1,i-1}=-b\hspace{0.45mm}q^{2i-d-1}M_{ii}.$ By this and
(\ref{e3.7}) the left-hand side of (\ref{e3.18}) equals
$M_{ii}.$ This shows (\ref{e3.18}) and hence (\ref{e3.2}).

Case (III): Using (\ref{e13}), (\ref{e3.4}) we find the right-hand side of (\ref{e3.2}) equals $M_{ij}$ times
\begin{align}\label{e3.14}
a^{-1}b^{-1}~\frac{q^{2d-2j-1}+q^{2j+1}-q^{2d+1}-q^{-1}}{q^{2j+2}-q^{2i}}+c~\frac{q^{d+2j+2}+q^{2j-d}-q^{d}-q^{2j+2i-d}}{q^{2j+2}-q^{2i}}+c^{-1}q^{d-2j}.
\end{align}
Using (\ref{e3.9}) and Lemma~\ref{lem3.4} we find the left-hand
side of (\ref{e3.2}) equals $M_{ij}$ times (\ref{e3.14}). This shows
(\ref{e3.2}).



Case (IV): In this case (\ref{e3.2}) reduces to
\begin{align}\label{e3.10}
(A^{\varepsilon\natural})_{00} M_{0d}+(A^{\varepsilon\natural})_{01} M_{1d}=\theta^\varepsilon_d M_{0d}.
\end{align}
Put $(i,j)=(0,d)$ in (\ref{e4.2}) and get
\begin{align*}
M_{1d}=\frac{a}{(q-q^{-1})(q^{d-1}-abc)}~M_{0d}.
\end{align*}
Using this along with (\ref{e3.8}) at $i=1$ and (\ref{e3.9}) at
$i=0,$ we find the left-hand side of (\ref{e3.10}) equals the
right-hand side of (\ref{e3.10}). This shows (\ref{e3.10}) and
hence (\ref{e3.2}). We have verified (\ref{e3.2}) in each of the
cases (I)--(IV). The result follows.
\end{proof}

\begin{cor}\label{cor3.1}
Let $\{\theta_i^\varepsilon\}^d_{i=0}$ denote the scalars from {\rm (\ref{e13})}. Then the roots of the
characteristic polynomial of $A^\varepsilon$ are
$\{\theta_i^\varepsilon\}^d_{i=0}.$
\end{cor}
\begin{proof}
Immediate from Proposition~\ref{prop3.1}.
\end{proof}

\begin{cor}\label{thm3.1}
Let $\{\theta_i^\varepsilon\}^d_{i=0}$ denote the scalars from {\rm (\ref{e13})}. Then the  following {\rm (i)}--{\rm (iii)} are equivalent.
\begin{enumerate}
\item[{\rm (i)}] $A^\varepsilon$ is multiplicity-free.

\item[{\rm (ii)}] $\{\theta_i^\varepsilon\}^d_{i=0}$ are mutually
distinct.

\item[{\rm (iii)}] $c^2$ is not among
$q^{2d-2},q^{2d-4},\ldots,q^{2-2d}.$
\end{enumerate}
\end{cor}
\begin{proof}
(i) $\Leftrightarrow$ (ii): By Corollary~\ref{cor3.1}.

(ii) $\Leftrightarrow$ (iii): By Lemma~\ref{lem1.6} and (RQRAC2).
\end{proof}

\section{The condition for $(A,A^*,A^\varepsilon)$ to be a Leonard triple}
Let $(A,A^*)$ denote a Leonard pair on $V$ that has QRacah type.
Fix $(a,b,c;q)\in QRAC_{red}$ which corresponds to $(A,A^*),$ and
let $A^\varepsilon \in {\rm End}(V)$ be the corresponding element
from Theorem~\ref{thm2.3}. In Corollary~\ref{thm3.1} we found necessary and sufficient conditions for $A^\varepsilon$ to be multiplicity-free. In this section we show that $A^\varepsilon$ is multiplicity-free if and only if $(A,A^*,A^\varepsilon)$ is a Leonard triple.

\begin{prop}\label{prop4.1}
We have
\begin{align}\label{e4.9}
A^{*\rho}=\left(
\begin{array}{cccccc}
\theta_0^* &\varphi_1^\varepsilon & & & &{\bf 0}\\
  &\theta_1^* &\varphi_2^\varepsilon & & &\\
  & &\theta_2^* &\cdot & &\\
  & & &\cdot &\cdot &\\
  & & & &\cdot &\varphi_d^\varepsilon\\
{\bf 0} & & & & &\theta_d^*
\end{array}
\right),
\end{align}
where $\rho$ is from Definition~\ref{defn2.1}, the scalars
$\{\theta_i^*\}^d_{i=0}$ are from {\rm (\ref{e12})} and
\begin{align}\label{e4.10}
\varphi_i^\varepsilon=b^{-1}c^{-1}q^{d+1}(q^{i}-q^{-i})(q^{i-d-1}-q^{d-i+1})(q^{-i}-abcq^{i-d-1})(q^{-i}-a^{-1}bcq^{i-d-1})
\end{align}
for $1\leq i\leq d.$
\end{prop}
\begin{proof}
Let $B$ denote the matrix on the right in (\ref{e4.9}). To show
that $A^{*\rho}=B,$ it suffices to show that $A^{*\natural} M=MB.$ Recall that $A^{*\natural}$ is the matrix on the right in (\ref{e3.12}).
In order to show $A^{*\natural} M=MB,$ for $0\leq i,j\leq d$ we show that the $(i,j)$-entry
of $A^{*\natural} M$ equals the $(i,j)$-entry of $MB.$ In other
words, it suffices to show
\begin{align}\label{e4.1}
(\theta^*_i-\theta^*_j)M_{i,j}+\varphi_{i+1}M_{i+1,j}=\varphi_j ^\varepsilon M_{i,j-1}.
\end{align}
To verify (\ref{e4.1}) we consider the following two cases: (I)
$j-i\leq 0;$ (II) $j-i>0.$

Case (I): Each summand in (\ref{e4.1}) is zero so
(\ref{e4.1}) holds.

Case (II): Evaluating the left-hand side of (\ref{e4.1}) using
Lemma~\ref{lem1.2} and (\ref{e9}), (\ref{e4.2}) we find that it
equals $M_{ij}$ times
\begin{align}\label{e4.11}
(q^{i-j}-q^{j-i})(b\hspace{0.45mm}q^{i+j-d}-ac^{-1}q^{i-j+1}).
\end{align}
Evaluating the right-hand side of (\ref{e4.1}) using (\ref{e4.3}), (\ref{e4.10}) we find that it also equals $M_{ij}$ times (\ref{e4.11}). Therefore (\ref{e4.1}) holds. We have verified (\ref{e4.1}) in the cases (I), (II). The result follows.
\end{proof}

\begin{lem}\label{lem4.2}
Assume that $A^\varepsilon$ is multiplicity-free. For $0\leq i\leq d$ let $\theta_i^\varepsilon$ denote the scalar from {\rm (\ref{e13})} and let $E_i^\varepsilon$ denote the primitive idempotent of
$A^\varepsilon$ associated with $\theta_i^\varepsilon.$ Then $(A^*;\{E_i^*\}^d_{i=0};A^\varepsilon;\{E_i^\varepsilon\}^d_{i=0})$
is a Leonard system of QRacah type and corresponds to $(b,c,a;q).$
\end{lem}
\begin{proof}
By Corollary~\ref{thm3.1}, $c^2$ is not among
$q^{2d-2},q^{2d-4},\ldots,q^{2-2d}.$ By this and since $(a,b,c;q)\in QRAC_{red},$ we have $(x,y,z;q)\in QRAC_{red}$ for any permutation $x,y,z$ of $a,b,c.$ In particular
$(c,b,a;q)\in QRAC_{red}.$ Define
$$
\phi_i^\varepsilon=b^{-1}c\hspace{0.45mm}q^{d+1}(q^i-q^{-i})(q^{i-d-1}-q^{d-i+1})(q^{-i}-abc^{-1}q^{i-d-1})(q^{-i}-a^{-1}bc^{-1}q^{i-d-1})
$$
for $1\leq i\leq d.$ Observe
\begin{align}\label{e3.11}
(\{\theta_i^\varepsilon\}^d_{i=0},\{\theta_i^*\}^d_{i=0},\{\varphi_i^\varepsilon\}^d_{i=1},\{\phi_i^\varepsilon\}^d_{i=1})
\end{align}
is an element of $PA$-$QRAC$ that corresponds to  $(c,b,a;q).$
Note that $A^{\varepsilon\rho}$ is lower bidiagonal by
Proposition~\ref{prop3.1} and $A^{*\rho}$ is upper bidiagonal by
Proposition~\ref{prop4.1}. Moreover
\begin{align*}
(A^{\varepsilon\rho})_{ii}=\theta_i^\varepsilon,\qquad \qquad (A^{*\rho})_{ii}=\theta_i^* \qquad \qquad \hbox{$(0\leq i\leq d),$}\\
(A^{\varepsilon\rho})_{i,i-1}(A^{*\rho})_{i-1,i}=\varphi_i^\varepsilon \qquad \qquad \qquad \hbox{$(1\leq i\leq d).$}
\end{align*}
By this and Lemma~\ref{thm1.5},
$(A^{\varepsilon\rho};\{E_i^{\varepsilon\rho}\}^d_{i=0};A^{*\rho};\{E_i^{*\rho}\}^d_{i=0})$
is a Leonard system of QRacah type that has parameter array
(\ref{e3.11}). Therefore
$(A^{\varepsilon\rho};\{E_i^{\varepsilon\rho}\}^d_{i=0};A^{*\rho};\{E_i^{*\rho}\}^d_{i=0})$
corresponds to $(c,b,a;q).$ Since $\rho$ is a $\mathbb{K}$-algebra
isomorphism
$\Phi=(A^\varepsilon;\{E_i^\varepsilon\}^d_{i=0};A^*;\{E_i^*\}^d_{i=0})$
is a Leonard system of QRacah type that corresponds to
$(c,b,a;q).$ Therefore
$\Phi^*=(A^*;\{E_i^*\}^d_{i=0};A^\varepsilon;\{E_i^\varepsilon\}^d_{i=0})$
is a Leonard system of QRacah type. By Lemma~\ref{lem2.4} the Leonard system
$\Phi^*$ corresponds to $(b^{-1},c^{-1},a^{-1};q^{-1}),$ and also
corresponds to $(b,c,a;q)$ by Corollary~\ref{lem2.8}.
\end{proof}

\begin{lem}\label{lem4.4}
Assume that $A^\varepsilon$ is multiplicity-free. For $0\leq i\leq d$ let $\theta_i^\varepsilon$ denote the scalar from {\rm (\ref{e13})} and let $E_i^\varepsilon$ denote the primitive idempotent of $A^\varepsilon$ associated with $\theta_i^\varepsilon.$ Then  $(A^\varepsilon;\{E_i^\varepsilon\}^d_{i=0};A;\{E_i\}^d_{i=0})$ is a Leonard system of QRacah type and corresponds to $(c,a,b;q).$
\end{lem}
\begin{proof}
By Lemma~\ref{lem4.2} the Leonard pair $(A^*,A^\varepsilon)$ has QRacah type and corresponds to $(b,c,a;q)\in QRAC_{red}.$ Now by Theorem~\ref{thm2.3} there exists an element $A^\vee$ in ${\rm End}(V)$ such that $A^*,$ $A^\varepsilon,$ $A^\vee$ satisfy the $\mathbb{Z}_3$-symmetric Askey-Wilson relations, one of which is
$$
\frac{qA^*A^\varepsilon-q^{-1}A^\varepsilon A^*}{q^2-q^{-2}}+A^\vee=\frac{(b+b^{-1})(c+c^{-1})+(a+a^{-1})(q^{d+1}+q^{-d-1})}{q+q^{-1}}~I.
$$
Comparing this with (\ref{e15}) we find $A^\vee=A.$ We now apply Lemma~\ref{lem4.2} to $A'=A^*,$ ${A^*}'=A^\varepsilon,$ ${A^\varepsilon}'=A,$ $a'=b,$ $b'=c,$ $c'=a,$ and obtain the desired result.
\end{proof}

\begin{cor}\label{lem4.3}
Assume that $A^\varepsilon$ is multiplicity-free. For
$0\leq i\leq d$ let $\theta_i^\varepsilon$ denote the scalar from {\rm (\ref{e13})} and let $E_i^\varepsilon$ denote the primitive
idempotent of $A^\varepsilon$ associated with
$\theta_i^\varepsilon.$ Then $(A;\{E_i\}^d_{i=0};A^*;$
$\{E_i^*\}^d_{i=0};$ $A^\varepsilon;\{E_i^\varepsilon\}^d_{i=0})$
is a Leonard triple system of QRacah type.
\end{cor}
\begin{proof}
Immediate from Lemma~\ref{lem1.1}, Definition~\ref{defn1.12},
Lemma~\ref{lem4.2}, Lemma~\ref{lem4.4}.
\end{proof}

\begin{thm}\label{thm4.1}
With reference to Theorem~\ref{thm2.3}, the following {\rm (i)}--{\rm (iii)} are equivalent.
\begin{enumerate}
\item[{\rm (i)}] $(A,A^*,A^\varepsilon)$ is a Leonard triple on $V.$

\item[{\rm (ii)}] $A^\varepsilon$ is multiplicity-free.

\item[{\rm (iii)}] $c^2$ is not among
$q^{2d-2},q^{2d-4},\ldots,q^{2-2d}.$
\end{enumerate}
Suppose {\rm (i)}--{\rm (iii)} hold. For $0\leq i\leq d$ let $\theta_i^\varepsilon$ denote the scalar from {\rm (\ref{e13})}. Then the third eigenvalue sequence of each Leonard triple system associated with $(A,A^*,A^\varepsilon)$ is eithr $\{\theta_i^\varepsilon\}^d_{i=0}$ or $\{\theta_{d-i}^\varepsilon\}^d_{i=0}.$
\end{thm}
\begin{proof}
(i) $\Rightarrow$ (ii): By Lemma~\ref{lem1.0}.

(ii) $\Rightarrow$ (i): By Corollary~\ref{lem4.3}.

(ii) $\Leftrightarrow$ (iii): By Corollary~\ref{thm3.1}.

Suppose (i)--(iii) hold. By Corollary~\ref{lem4.3} there is a Leonard triple system $\Psi$ associated with $(A,A^*,A^\varepsilon)$ that has the third eigenvalue sequence $\{\theta_i^\varepsilon\}^d_{i=0}.$ By Definition~\ref{defn2.5} the associate class for $(A,A^*,A^\varepsilon)$ is the $(\mathbb{Z}_2)^3$-orbit containing $\Psi.$ Now the last assertion follows from the table above Definition~\ref{defn3.1}.
\end{proof}

\section{A set $T$-$QRAC_{red}$}

The following definition is motivated by Theorem~\ref{thm4.1}.

\begin{defn}\label{defn1.13}
Let $T$-$QRAC_{red}=T$-$QRAC_{red}(d,\mathbb{K})$ denote the set
of all $4$-tuples $(a,b,c;q)$ of scalars in $\mathbb{K}$ that
satisfy the following conditions (T-RQRAC1)--(T-RQRAC4).
\begin{description}
\item[{\rm (T-RQRAC1)}] $a\not=0,$ $b\not=0,$ $c\not=0,$ $q\not=0.$

\item[{\rm (T-RQRAC2)}] $q^{2i}\not=1$ for $1\leq i\leq d.$

\item[{\rm (T-RQRAC3)}] None of $a^2,$ $b^2,$ $c^2$ is among $q^{2d-2},q^{2d-4},\ldots,q^{2-2d}.$

\item[{\rm (T-RQRAC4)}] None of $abc,$ $a^{-1}bc,$ $ab^{-1}c,$ $abc^{-1}$ is among $q^{d-1},q^{d-3},\ldots,q^{1-d}.$
\end{description}
\end{defn}

We observe that $T$-$QRAC_{red}$ is a subset of the set
$QRAC_{red}$ from Definition~\ref{defn1.8}.  Recall the $D_4$
action on $QRAC_{red},$ from Lemma~\ref{lem2.3}.

\begin{lem}\label{lem1.18}
The set $T$-$QRAC_{red}$ is closed under the action of $D_4$ on $QRAC_{red}.$
\end{lem}
\begin{proof}
Let $(a,b,c;q)\in T$-$QRAC_{red}.$ It is routine to check each of
$(a,b,c;q)^*,(a,b,c;q)^\downarrow,(a,b,c;q)^\Downarrow$ is
contained in $T$-$QRAC_{red}.$ The result follows since
$*,\downarrow,\Downarrow$ generate $D_4.$
\end{proof}

\section{Classification of the Leonard triple systems of QRacah type}

In this section we classify up to isomorphism the Leonard triple systems of QRacah type. We do this as follows. Recall from Definition~\ref{defn3.3} that $T$-$QRAC$ is the set of isomorphism classes of Leonard triple systems that have QRacah type. Recall the set $T$-$QRAC_{red}$ from Definition~\ref{defn1.13}. We display a bijection $\pi:T$-$QRAC_{red}\to T$-$QRAC.$

\begin{defn}\label{defn4.1}
Define a map $\pi:T$-$QRAC_{red}\rightarrow T$-$QRAC$ as follows.
Let $(a,b,c;q)\in T$-$QRAC_{red}.$ By Theorem~\ref{thm4.1} along with Lemmas~\ref{lem4.2} and \ref{lem4.4} there exists a Leonard triple system of QRacah type
$$
\Psi=(A;\{E_i\}^d_{i=0};A^*;\{E_i^*\}^d_{i=0};A^\varepsilon;\{E_i^\varepsilon\}^d_{i=0})
$$
that satisfies (i)--(iv) below.
\begin{enumerate}
\item[(i)] The Leonard system $(A;\{E_i\}^d_{i=0};A^*;\{E_i^*\}^d_{i=0})$ corresponds to $(a,b,c;q).$

\item[(ii)] The Leonard system $(A^*;\{E_i^*\}^d_{i=0};A^\varepsilon;\{E_i^\varepsilon\}^d_{i=0})$ corresponds to $(b,c,a;q).$

\item[(iii)] The Leonard system $(A^\varepsilon;\{E_i^\varepsilon\}^d_{i=0};A;\{E_i\}^d_{i=0})$ corresponds to $(c,a,b;q).$

\item[(iv)] The elements $A,A^*,A^\varepsilon$ satisfy the $\mathbb{Z}_3$-symmetric Askey-Wilson relations (\ref{e15})--(\ref{e14}).

\end{enumerate}
Observe that $\Psi$ is unique up to isomorphism of Leonard triple systems. The map $\pi$ sends $(a,b,c;q)$ to the isomorphism class of $\Psi$ in the set of all Leonard triple systems.
\end{defn}

\begin{defn}\label{defn4.3}
Referring to Definition~\ref{defn4.1} we say that $\Psi$ and $(a,b,c;q)$ {\it correspond via $\pi.$} Similarly, we say that the Leonard triple $(A,A^*,A^\varepsilon)$ and $(a,b,c;q)$ {\it correspond via $\pi.$}
\end{defn}

We are going to show that $\pi$ is a bijection. Before we do this, we give a concrete description of $\pi$ using matrices. Pick $(a,b,c;q)$ in $T$-$QRAC_{red}.$
Define
$\{\theta_i\}^d_{i=0},\{\theta_i^*\}^d_{i=0}$ by (\ref{e7}), (\ref{e8}) and $\{\varphi_i\}^d_{i=1},\{\phi_i\}^d_{i=1}$ by (\ref{e9}), (\ref{e10}).
By Lemma~\ref{lem1.15},
\begin{align}\label{e6.18}
(\{\theta_i\}^d_{i=0},\{\theta_i^*\}^d_{i=0},\{\varphi_i\}^d_{i=1},\{\phi_i\}^d_{i=1})
\end{align}
is a parameter array of QRacah type and corresponds to $(a,b,c;q).$ Let $A$ (resp. $A^*$) denote the lower
bidiagonal (resp. upper bidiagonal) matrix in ${\rm
Mat}_{d+1}(\mathbb{K})$ with entries
\begin{align*}
&A_{ii}=\theta_i, \qquad  A^*_{ii}=\theta_i^* \qquad \quad \hbox{$(0\leq i\leq d),$}\\
&A_{i,i-1}=1, \qquad A^*_{i-1,i}=\varphi_i \qquad \quad  \hbox{$(1\leq i\leq d).$}
\end{align*}
For $0\leq i\leq d$ let $E_i, E_i^*$ denote the
primitive idempotents of $A, A^*$ associated with
$\theta_i, \theta_i^*$ respectively. By Lemma~\ref{thm1.5}, $(A;\{E_i\}^d_{i=0};A^*;\{E_i^*\}^d_{i=0})$ is a Leonard system that has parameter array (\ref{e6.18}). Therefore the Leonard system $(A;\{E_i\}^d_{i=0};A^*;\{E_i^*\}^d_{i=0})$ corresponds to $(a,b,c;q).$ Define
$$
A^\varepsilon=\frac{q^{-1}A^*A-qAA^*}{q^2-q^{-2}}+\frac{(a+a^{-1})(b+b^{-1})+(c+c^{-1})(q^{d+1}+q^{-d-1})}{q+q^{-1}}~I.
$$
In concrete terms $A^\varepsilon$ is the tridiagonal matrix in ${\rm
Mat}_{d+1}(\mathbb{K})$ with entries
$$
A^{\varepsilon}_{ii}=\frac{(a+a^{-1})(b+b^{-1})+(c+c^{-1})(q^{d+1}+q^{-d-1})}{q+q^{-1}}-\frac{\theta_i\theta_i^*}{q+q^{-1}}-\frac{q\varphi_i-q^{-1}\varphi_{i+1}}{q^2-q^{-2}}
$$
for $0 \leq i \leq d$ and
$$
A^{\varepsilon}_{i,i-1}=\frac{q^{-1}\theta_{i}^*-q\theta_{i-1}^*}{q^2-q^{-2}},\qquad
\quad
A^{\varepsilon}_{i-1,i}=\varphi_i\hspace{0.5mm}\frac{q^{-1}\theta_i-q\theta_{i-1}}{q^2-q^{-2}}
$$
for $1 \leq i \leq d.$ By Theorem~\ref{thm2.3} the matrices $A,A^*,A^\varepsilon$ satisfy (\ref{e15})--(\ref{e14}). By Corollaries~\ref{cor3.1}
and \ref{thm3.1} the matrix $A^\varepsilon$ is multiplicity-free
with eigenvalues
$$
\theta_i^\varepsilon=cq^{2i-d}+c^{-1}q^{d-2i} \qquad \quad (0\leq i\leq d).
$$
For $0\leq i\leq d$ let $E_i^\varepsilon$ denote the
primitive idempotent of $A^\varepsilon$ associated with
$\theta_i^\varepsilon.$ By Corollary~\ref{lem4.3} the sequence $\Psi=(A;\{E_i\}^d_{i=0};A^*;\{E_i^*\}^d_{i=0};A^\varepsilon;\{E_i^\varepsilon\}^d_{i=0})$ is a Leonard triple system of QRacah type. By Lemma~\ref{lem4.2} the Leonard system $(A^*;\{E_i^*\}^d_{i=0};A^\varepsilon;\{E_i^\varepsilon\}^d_{i=0})$ corresponds to $(b,c,a;q).$ By Lemma~\ref{lem4.4} the Leonard system $(A^\varepsilon;\{E_i^\varepsilon\}^d_{i=0};A;\{E_i\}^d_{i=0})$ corresponds to $(c,a,b;q).$ Therefore the map $\pi$ sends $(a,b,c;q)$ to the isomorphism class of $\Psi.$ This concludes our concrete description of $\pi.$

\medskip

In our classification we will make use of the following result.

\begin{thm}\label{thm6.1}{\rm (\cite[Theorem~3.2]{Nom&ter2007}).} Let $(A;\{E_i\}^d_{i=0};A^*;\{E_i^*\}^d_{i=0})$ denote a Leonard system on $V.$ Let $\mathcal{X}$ denote the $\mathbb{K}$-subspace of ${\rm End}(V)$ consisting of the $X\in {\rm End}(V)$ such that both
\begin{align*}
E_iXE_j&=0 \qquad \qquad \hbox{if $|i-j|>1,$}\\
E_i^*XE_j^*&=0 \qquad \qquad \hbox{if $|i-j|>1$}
\end{align*}
for $0\leq i,j\leq d.$ Then the space $\mathcal{X}$ is spanned by
\begin{align}\label{e5.1}
I,~A,~A^*,~AA^*,~A^*A.
\end{align}
Moreover {\rm (\ref{e5.1})} is a basis for $\mathcal{X}$ provided $d\geq 2.$
\end{thm}

We now give our classification of the Leonard triple systems of QRacah type.

\begin{thm}\label{thm6.2}
The map $\pi$ from Definition~\ref{defn4.1} is a bijection.
\end{thm}
\begin{proof} Let $\Psi=(A;\{E_i\}^d_{i=0};A^*;\{E_i^*\}^d_{i=0};A^\varepsilon;\{E_i^\varepsilon\}^d_{i=0})$ denote a Leonard triple system on $V$ that has QRacah type. It suffices to show that there exists a unique element of $T$-$QRAC_{red}$ that corresponds to $\Psi$ via $\pi.$ Let $\{\theta_i\}^d_{i=0},$ $\{\theta_i^*\}^d_{i=0},$ $\{\theta_i^\varepsilon\}^d_{i=0}$ denote the first, second, third eigenvalue sequences of $\Psi,$ respectively. By Definition~\ref{defn1.12} there exist nonzero $a,b,c,q\in \mathbb{K}$ that satisfy (\ref{e11})--(\ref{e13}).  Let
\begin{align*}
\Phi&=(A;\{E_i\}^d_{i=0};A^*;\{E_i^*\}^d_{i=0}),\\
\Phi'&=(A^*;\{E_i^*\}^d_{i=0};A^\varepsilon;\{E_i^\varepsilon\}^d_{i=0}),\\
\Phi''&=(A^\varepsilon;\{E_i^\varepsilon\}^d_{i=0};A;\{E_i\}^d_{i=0}).
\end{align*}
By construction each of $\Phi,$ $\Phi',$ $\Phi''$ is a Leonard system of QRacah
type. By Corollary~\ref{lem1.10} and Lemma~\ref{lem1.19} there exist nonzero scalars $x,y,z$ in $\mathbb{K}$ such that each of
$$
(a,b,x;q),\qquad (b,c,y;q), \qquad (c,a,z;q)
$$
is in  $QRAC_{red}$ and correspond to $\Phi,$ $\Phi',$ $\Phi''$ respectively. Moreover each of $x,y,z$ is unique up to inversion. By Theorem~\ref{thm6.1} there exist unique $e,f,f^*,g,g^*\in
\mathbb{K}$ such that
\begin{align}\label{e6.17}
A^\varepsilon=eI+fA+f^*A^*+gAA^*+g^*A^*A.
\end{align}
Let $\varphi_i$ $(1\leq i\leq d)$ denote the scalar in $\mathbb{K}$ obtained by changing $c$ to $x$ in (\ref{e9}). Then $\{\varphi_i\}^d_{i=1}$ is the first split sequence of $\Phi.$ Let
$\natural:{\rm End}(V)\rightarrow {\rm Mat}_{d+1}(\mathbb{K})$ denote the natural map for $\Phi.$ Recall that $A^\natural$ (resp. $A^{*\natural}$) is the matrix on the left (resp. right) in (\ref{e3.12}). Using this and (\ref{e6.17}) we find the matrix $A^{\varepsilon\natural}$ is tridiagonal with entries
\begin{align}
(A^{\varepsilon\natural})_{i,i-1}&=f+g\theta_{i-1}^*+g^*\theta_i^*, \label{e6.13}\\
(A^{\varepsilon\natural})_{i-1,i}&=\varphi_i(f^*+g\theta_{i-1}+g^*\theta_i) \label{e6.14}
\end{align}
for $1\leq i\leq d$ and
\begin{align}
(A^{\varepsilon\natural})_{ii}~~&=e+f\theta_i+f^*\theta_i^*+g(\theta_i\theta_i^*+\varphi_i)+g^*(\theta_i\theta_i^*+\varphi_{i+1}) \label{e6.15}
\end{align}
for $0\leq i\leq d.$ Applying Lemma~\ref{lem2.9}
to $(A^\varepsilon,A)$ and $(c,a,z;q)$  the Askey-Wilson relations for $(A^\varepsilon,A)$ are
\begin{align}
&A^2A^\varepsilon-(q^2+q^{-2}) AA^\varepsilon A+A^\varepsilon
A^2+(q^2-q^{-2})^2 A^\varepsilon =\omega A+\eta\hspace{0.45mm} I,\label{e6.9}\\
&A^{\varepsilon2}A-(q^2+q^{-2}) A^\varepsilon A A^\varepsilon+A
A^{\varepsilon2}+(q^2-q^{-2})^2 A =\omega A^\varepsilon+\eta^\varepsilon I,\label{e6.11}
\end{align}
where
\begin{align*}
\omega&=-(q-q^{-1})^2\big((c+c^{-1})(a+a^{-1})+(z+z^{-1})(q^{d+1}+q^{-d-1})\big),\\
\eta&=(q-q^{-1})(q^2-q^{-2})\big((a+a^{-1})(z+z^{-1})+(c+c^{-1})(q^{d+1}+q^{-d-1})\big),\\
\eta^\varepsilon&=(q-q^{-1})(q^2-q^{-2})\big((z+z^{-1})(c+c^{-1})+(a+a^{-1})(q^{d+1}+q^{-d-1})\big).
\end{align*}
In what follows, when
referring to (\ref{e6.11}), (\ref{e6.9}) we mean the
relations obtained by applying $\natural$
to (\ref{e6.9}), (\ref{e6.11}) respectively. For convenience (\ref{e6.13})--(\ref{e6.15}) will be used tacitly to evaluate $A^{\varepsilon\natural}.$

We now find the values of $e,$ $f,$ $f^*,$ $g,$ $g^*.$ Concerning the $(2,0)$-entry of either side of (\ref{e6.9}), the right-hand side is zero. We evaluate the left-hand side and by (RQRAC1),  (RQRAC2) some factors in the resulting equation are nonzero. Eliminating those factors we obtain
\begin{align}
e+(q^2+1+q^{-2})(aq^{2-d}+a^{-1}q^{d-2})f+\kappa g+\kappa g^*&=0, \label{e6.3}
\end{align}
where
$$
\kappa=(a+a^{-1})(b+b^{-1})+(x+x^{-1})(q^{d+1}+q^{-d-1}).
$$
Similarly we evaluate $(3,1)$-entry of either side of (\ref{e6.9}) and obtain
\begin{align}
e+(q^2+1+q^{-2})(aq^{4-d}+a^{-1}q^{d-4})f+\kappa g+\kappa g^*&=0. \label{e6.4}
\end{align}
Subtracting (\ref{e6.3}) from (\ref{e6.4}) we find
\begin{align}\label{e6.5}
(q^3-q^{-3})(aq^{3-d}-a^{-1}q^{d-3})f=0.
\end{align}
In the left-hand side of (\ref{e6.5}) the first term is nonzero by (RQRAC2) with $i=3$ and the second term is nonzero by (RQRAC3). Therefore $f=0.$ Now (\ref{e6.3}) becomes
\begin{align}
e=-\kappa(g+g^*). \label{e6.6}
\end{align}
Similarly, under the natural map for $\Phi^*$ consider the $(2,0)$-entry and $(3,1)$-entry of either side of the Askey-Wilson relation
$$
A^{*2}A^\varepsilon-(q^2+q^{-2}) A^*A^\varepsilon A^*+A^\varepsilon
A^{*2}+(q^2-q^{-2})^2 A^\varepsilon =\varpi A^*+\zeta^* I
$$
where
\begin{align*}
\varpi&=-(q-q^{-1})^2\big((b+b^{-1})(c+c^{-1})+(y+y^{-1})(q^{d+1}+q^{-d-1})\big),\\
\zeta^*&=(q-q^{-1})(q^2-q^{-2})\big((y+y^{-1})(b+b^{-1})+(c+c^{-1})(q^{d+1}+q^{-d-1})\big).
\end{align*}
One can show $f^*=0.$ We now show that
\begin{align}\label{e6.19}
g=-q^2g^* \qquad \hbox{or}\qquad  g=-q^{-2}g^*.
\end{align}
To do this we divide the argument into the three cases: (A) $q^4\not=-1;$ (B) $d\geq 4;$ (C) $d=3$ and $q^4=-1.$

Case (A): Concerning the $(3,0)$-entry of either side of (\ref{e6.11}), the right-hand side is zero. We evaluate the left-hand side and replace $e$ by (\ref{e6.6}). By (RQRAC1) and (RQRAC2) some factors in the resulting equation are nonzero. Eliminating those factors we obtain
\begin{align}\label{e6.7}
(q^2+q^{-2})(q^{-1}g+qg^*)(qg+q^{-1}g^*)=0.
\end{align}
In the left-hand side of (\ref{e6.7}) the first term is nonzero since $q^4\not=-1.$  Therefore (\ref{e6.19}) holds.

Case (B): In the left-hand side of (\ref{e6.7}) the first term is nonzero by (RQRAC2) with $i=4.$ Therefore (\ref{e6.19}) holds.

Case (C): Concerning the $(2,0)$-entry of either side of (\ref{e6.11}), the right-hand side is zero. We evaluate the left-hand side and replace $e$ by (\ref{e6.6}). By (RQRAC1) and (RQRAC2) some factors in the resulting equation are nonzero. Eliminating those factors we obtain
$$
\nu(qg+q^{-1}g^*)(q^{-1}g+qg^*)=0,
$$
where $\nu=2(a+a^{-1})-(b+b^{-1})(x+x^{-1}).$
Similarly, under the natural map for $\Phi^*$ we consider the $(2,0)$-entry of either side of the Askey-Wilson relation
$$
A^{\varepsilon2}A^*-(q^2+q^{-2}) A^\varepsilon A^* A^\varepsilon+A^*
A^{\varepsilon2}+(q^2-q^{-2})^2 A^* =\varpi A^\varepsilon+\zeta I
$$
where
$
\zeta=(q-q^{-1})(q^2-q^{-2})\big((c+c^{-1})(y+y^{-1})+(b+b^{-1})(q^{d+1}+q^{-d-1})\big),
$
and obtain
$$
\nu^*(qg+q^{-1}g^*)(q^{-1}g+qg^*)=0,
$$
where $\nu^*=2(b+b^{-1})-(a+a^{-1})(x+x^{-1}).$ To get (\ref{e6.19}) we show that
$\nu\not=0$ or $\nu^*\not=0$ by contradiction. By (RQRAC3) we have $a^2\not=-1$ and $b^2\not=-1.$ We use this to solve for $b,$ $x$ in
$\nu=0$ and $\nu^*=0,$ and get that $(b,x)$ is one of the following pairs:
$$
(a,1),\quad (a^{-1},1),\quad (-a,-1),\quad (-a^{-1},-1),
$$
any of which contradicts (RQRAC4). This proves (\ref{e6.19}).

Combining (\ref{e6.6}) with (\ref{e6.19}) we find
\begin{align}\label{e6.8}
e=\left\{
\begin{array}{ll}
\kappa q(q-q^{-1}) g^* \qquad &\hbox{if $g=-q^2g^*,$}\\
-\kappa q^{-1}(q-q^{-1}) g^*\qquad \qquad &\hbox{if $g=-q^{-2}g^*.$}
\end{array}
\right.
\end{align}
Concerning (\ref{e6.11}), the $(0,1)$-entry and $(1,0)$-entry of the right-hand side are equal to $\omega(A^{\varepsilon\natural})_{01}$ and $\omega(A^{\varepsilon\natural})_{10},$ respectively. Therefore the $(0,1)$-entry of the left-hand side multiplied by $(A^{\varepsilon\natural})_{10}$ minus the $(1,0)$-entry of the left-hand side multiplied by $(A^{\varepsilon\natural})_{01}$ is equal to zero. On the other hand,  we directly evaluate the difference and replace $e,$ $g$ by (\ref{e6.8}), (\ref{e6.19}) respectively. By (RQRAC1), (RQRAC2), (RQRAC4) some factors in the resulting equation are nonzero. Eliminating those factors we get
\begin{align}\label{e6.10}
\left\{
\begin{array}{ll}
g^*(g^*(q^2-q^{-2})+q^{-1})(g^*(q^2-q^{-2})-q^{-1})=0 \qquad \qquad &\hbox{if $g=-q^2g^*,$}\\
g^*(g^*(q^2-q^{-2})+q)(g^*(q^2-q^{-2})-q)=0\qquad &\hbox{if $g=-q^{-2}g^*.$}
\end{array}
\right.
\end{align}
If $g^*=0$ then $g=0$ by (\ref{e6.19}) and $e=0$ by (\ref{e6.8}) hence $A^\varepsilon=0,$ a contradiction. Therefore $g^*\not=0$ and so (\ref{e6.10}) yields
\begin{align}\label{e6.12}
\left\{
\begin{array}{ll}
g^*={\displaystyle \frac{q^{-1}}{q^2-q^{-2}}} \quad {\rm or} \quad g^*={\displaystyle \frac{-q^{-1}}{q^2-q^{-2}}} \qquad \qquad &\hbox{if $g=-q^2g^*,$}\\
g^*={\displaystyle \frac{q}{q^2-q^{-2}}} \quad {\rm or} \quad g^*={\displaystyle \frac{-q}{q^2-q^{-2}}} \qquad &\hbox{if $g=-q^{-2}g^*.$}
\end{array}
\right.
\end{align}
Combining  (\ref{e6.19}) with (\ref{e6.8}) and (\ref{e6.12}) we find $(e,g,g^*)$ is one
of the following sequences:
\begin{align*}
&\hbox{(I)}\quad {\displaystyle \bigg(\frac{\kappa}{q+q^{-1}},\frac{-q}{q^2-q^{-2}},\frac{q^{-1}}{q^2-q^{-2}}
\bigg)}; \qquad \quad
&\hbox{(II)} \quad  {\displaystyle\bigg(\frac{-\kappa}{q+q^{-1}},\frac{q}{q^2-q^{-2}},\frac{-q^{-1}}{q^2-q^{-2}}
\bigg);}\\
&\hbox{(III)}\quad {\displaystyle\bigg(\frac{\kappa}{q+q^{-1}},\frac{q^{-1}}{q^2-q^{-2}},\frac{-q}{q^2-q^{-2}}\bigg)};
&\hbox{(IV)}\quad
 {\displaystyle\bigg(\frac{-\kappa}{q+q^{-1}},\frac{-q^{-1}}{q^2-q^{-2}},\frac{q}{q^2-q^{-2}}\bigg)}.
\end{align*}
By Corollary~\ref{lem2.8} the Leonard system $\Phi$ corresponds to
\begin{align}\label{e6.16}
\begin{split}
\begin{array}{ll}
(a,b,x;q), \qquad  \qquad &((-1)^da,(-1)^db,(-1)^{d+1}x; -q),\\
(a,b,x^{-1};q), \qquad \qquad &((-1)^da, (-1)^db, (-1)^{d+1}x^{-1}; -q),\\
(a^{-1},b^{-1},x^{-1};q^{-1}), \qquad  \qquad&((-1)^da^{-1},(-1)^db^{-1},(-1)^{d+1}x^{-1};-q^{-1}),\\
(a^{-1},b^{-1},x;q^{-1}), \qquad  \qquad &((-1)^da^{-1},(-1)^db^{-1},(-1)^{d+1}x;-q^{-1})
\end{array}
\end{split}
\end{align}
and no other elements of $QRAC_{red}.$ We now divide the argument into the cases (I)--(IV).

Case (I): Applying Theorem~\ref{thm2.3} to $(a,b,x;q)$ or $(a,b,x^{-1};q),$ the corresponding element in ${\rm End}(V)$ is exactly $A^\varepsilon.$ By Theorem~\ref{thm4.1} the sequences $(a,b,x;q),$ $(a,b,x^{-1};q)$ are in $T$-$QRAC_{red}.$ Moreover  $\theta_i^\varepsilon=xq^{2i-d}+x^{-1}q^{d-2i}$ for $0\leq i\leq d$ or $\theta_i^\varepsilon=xq^{d-2i}+x^{-1}q^{2i-d}$ for $0\leq i\leq d,$ and this implies $x=c$ or $x=c^{-1}$ respectively. Therefore $(a,b,c;q)$ is in $T$-$QRAC_{red}$ and corresponds to $\Phi.$ Moreover $A^\varepsilon$ is the corresponding element in ${\rm End}(V)$ from Theorem~\ref{thm2.3}. By Lemma~\ref{lem4.2} and Lemma~\ref{lem4.4} the Leonard systems $\Phi'$ and $\Phi''$ correspond to $(b,c,a;q)$ and $(c,a,b;q),$ respectively. We have shown that $\Psi$ and $(a,b,c;q)$ correspond via $\pi.$ It is routine to check that each sequence in (\ref{e6.16}) doesn't correspond to $\Psi$ via $\pi$ other than $(a,b,c;q).$ Therefore
$
(a,b,c;q)
$
is the unique element of $T$-$QRAC_{red}$ that corresponds to $\Psi$ via $\pi.$

Case (II): Applying Theorem~\ref{thm2.3} to $((-1)^da,(-1)^db,(-1)^{d+1}x;-q)$ or $((-1)^da,(-1)^db,(-1)^{d+1}$ $x^{-1};-q),$ the corresponding element in ${\rm End}(V)$ is exactly $A^\varepsilon.$ By the similar argument as case (I) we find $x=-c$ or $x=-c^{-1}$ and
$
((-1)^da,(-1)^db,(-1)^dc;-q)
$
is the unique element of $T$-$QRAC_{red}$ that corresponds to $\Psi$ via $\pi.$

Case (III): Applying Theorem~\ref{thm2.3} to $(a^{-1},b^{-1},x^{-1};q^{-1})$ or $(a^{-1},b^{-1},x;q^{-1}),$ the corresponding element in ${\rm End}(V)$ is exactly $A^\varepsilon.$ By the similar argument as case (I) we find $x=c$ or $x=c^{-1}$ and $(a^{-1},b^{-1},c^{-1};q^{-1})$ is the unique element of $T$-$QRAC_{red}$ that corresponds to $\Psi$ via $\pi.$

Case (IV): Applying Theorem~\ref{thm2.3} to $((-1)^da^{-1},(-1)^db^{-1},(-1)^{d+1}x^{-1};-q^{-1})$ or $((-1)^da^{-1},(-1)^d$ $b^{-1},(-1)^{d+1}x;-q^{-1}),$ the corresponding element in ${\rm End}(V)$ is exactly $A^\varepsilon.$ By the similar argument as (I) we find $x=-c$ or $x=-c^{-1}$ and $((-1)^da^{-1},(-1)^db^{-1},(-1)^dc^{-1};-q^{-1})$ is the unique element of $T$-$QRAC_{red}$ that corresponds to $\Psi$ via $\pi.$ We have completed the argument for the cases (I)--(IV). The result follows.
\end{proof}

In Theorem~\ref{thm6.2} we showed that $\pi$ is a bijection. We now describe $\pi^{-1}.$

\begin{lem}\label{lem6.1}
Let $\Psi=(A;\{E_i\}^d_{i=0};A^*;\{E_i^*\}^d_{i=0};A^\varepsilon;\{E_i^\varepsilon\}^d_{i=0})$ denote a Leonard triple system of QRacah type. Let $\{\theta_i\}^d_{i=0}, \{\theta_i^*\}^d_{i=0}, \{\theta_i^\varepsilon\}^d_{i=0}$ denote the first, second, and third eigenvalue sequences for $\Psi,$ respectively. Let $(a,b,c;q)$ denote the preimage with respect to $\pi,$ for the isomorphism class of $\Psi.$ Then $(a,b,c;q)$ is the unique sequence of scalars in $\mathbb{K}$ that satisfies {\rm (\ref{e15})--(\ref{e14})} and {\rm (\ref{e11})--(\ref{e13}).}
\end{lem}
\begin{proof}
We first show that $(a,b,c;q)$ satisfies (\ref{e15})--(\ref{e14}) and (\ref{e11})--(\ref{e13}).
The sequence $(a,b,c;q)$ satisfies (\ref{e11}), (\ref{e12}) by Definition~\ref{defn4.1}(i) and satisfies (\ref{e13}) by Definition~\ref{defn4.1}(ii). The sequence $(a,b,c;q)$ satisfies (\ref{e15})--(\ref{e14}) by Definition~\ref{defn4.1}(iv). Next we show that $(a,b,c;q)$ are uniquely determined by (\ref{e15})--(\ref{e14}) and (\ref{e11})--(\ref{e13}). Suppose we are given a sequence of scalars $(x,y,z;t)$ in $\mathbb{K}$ that satisfies
\begin{align}
\frac{tA^*A^\varepsilon-t^{-1}A^\varepsilon A^*}{t^2-t^{-2}}+A&=\frac{(y+y^{-1})(z+z^{-1})+(x+x^{-1})(t^{d+1}+t^{-d-1})}{t+t^{-1}}~I\label{et-15}
\end{align}
and
\begin{align*}
\theta_i=xt^{2i-d}+x^{-1}t^{d-2i}\qquad \quad \hbox{$(0\leq i\leq d),$}\\
\theta_i^*=yt^{2i-d}+y^{-1}t^{d-2i}\qquad \quad \hbox{$(0\leq i\leq d),$}\\
\theta_i^\varepsilon=zt^{2i-d}+z^{-1}t^{d-2i}\qquad \quad \hbox{$(0\leq i\leq d).$}
\end{align*}
We show $(x,y,z;t)=(a,b,c;q).$ By Theorem~\ref{thm6.1} the elements $A^*A^\varepsilon,$ $A^\varepsilon A^*,$ $I$
are linearly independent. Comparing (\ref{e15}), (\ref{et-15}) in this light, we find
\begin{align*}
t^{-1}(t^2-t^{-2})&= q^{-1}(q^2-q^{-2}),\\
t(t^2-t^{-2})&= q(q^2-q^{-2}).
\end{align*}
Solving the above two equations for $t$ we find $t=q.$ Now by Lemma~\ref{lem1.7}, $x=a.$ Similarly $y=b$ and $z=c.$ The result follows.
\end{proof}

\section{Twin pairs of Leonard triple systems}

\begin{defn}\label{defn7.1}
Let $\Psi$ and $\Psi'$ denote Leonard triple systems over $\mathbb{K}.$
We say that $\Psi$ and $\Psi'$ are {\it twins} whenever
\begin{enumerate}
\item[{\rm (i)}] $\Psi$ and $\Psi'$ have the same first eigenvalue sequence;

\item[{\rm (ii)}] $\Psi$ and $\Psi'$ have the same second eigenvalue sequence;

\item[{\rm (iii)}] $\Psi$ and $\Psi'$ have the same third eigenvalue sequence.
\end{enumerate}
\end{defn}

As we will see, there exist twin pairs of Leonard triple systems that are not isomorphic.

\begin{lem}\label{lem7.1}
Let $\Psi=(A;\{E_i\}^d_{i=0};A^*;\{E_i^*\}^d_{i=0};A^\varepsilon;\{E_i^\varepsilon\}^d_{i=0})$ denote a Leonard triple system on $V$ that has QRacah type.
Let $\dag$ denote the antiautomorphism of ${\rm End}(V)$ that corresponds to the Leonard pair $(A,A^*),$ in the sense of Definition~\ref{defn1.4}. Suppose $(a,b,c;q) \in T$-$QRAC_{red}$ corresponds to $\Psi$ via $\pi.$ Then $(a^{-1},b^{-1},c^{-1};q^{-1}) \in T$-$QRAC_{red}$ corresponds to $\Psi^\dag$ via $\pi.$ Moreover $\Psi$ and $\Psi^\dag$ are nonisomorphic twins.
\end{lem}
\begin{proof}
By construction $\Psi$ and $\Psi^\dag$ are twins. It is routine to verify that $(a^{-1},b^{-1},c^{-1};q^{-1})$ is in $T$-$QRAC_{red}$ according to Definition~\ref{defn1.13} and $(a^{-1},b^{-1},c^{-1};q^{-1})$ corresponds to $\Psi^\dag$ via $\pi$ according to Definition~\ref{defn4.1}. By (T-RQRAC2) with $i=1$ the sequences $(a,b,c;q)$ and $(a^{-1},b^{-1},c^{-1};q^{-1})$ are different. Therefore $\Psi$ and $\Psi^\dag$ are not isomorphic by Theorem~\ref{thm6.2}. The result follows.
\end{proof}

In this section we classify up to isomorphism all the twin pairs of Leonard triple systems that have QRacah type. To obtain the classification we use the bijection $\pi$ from Definition~\ref{defn4.1}.

\begin{defn}\label{defn7.2}
Let $(a,b,c;q)$ and $(a',b',c';q')$ denote two elements of $T$-$QRAC_{red}.$ We call these elements {\it twins} whenever the corresponding Leonard triple systems via $\pi$ are twins in the sense of Definition~\ref{defn7.1}.
\end{defn}

Observe that twin is an equivalence relation. We now describe the equivalence classes of the twin relation.

\begin{thm}\label{prop8.1}
Let $(a,b,c;q)$ denote an element of $T$-$QRAC_{red}.$ Then the twins of $(a,b,c;q)$ are
displayed below. There are two cases.
\begin{enumerate}
\item[{\rm (i)}] Assume that none of $abc,$ $a^{-1}bc,$ $ab^{-1}c,$ $abc^{-1}$ is among $-q^{d-1}, -q^{d-3},\ldots , -q^{1-d}.$ Then the twins of $(a,b,c;q)$ are
\begin{align}\label{e7.1}
\begin{array}{ll}
(a,b,c;q), \qquad \qquad &((-1)^da,(-1)^db,(-1)^dc;-q),\\
(a^{-1},b^{-1},c^{-1};q^{-1}), \qquad \qquad &((-1)^da^{-1},(-1)^db^{-1},(-1)^dc^{-1};-q^{-1}).
\end{array}
\end{align}
{\rm (}If $\mathbb{K}$ has characteristic two, then we interpret {\rm (\ref{e7.1})} having
only two elements{\rm )}.

\item[{\rm (ii)}] Assume that some of $abc,$ $a^{-1}bc,$ $ab^{-1}c,$ $abc^{-1}$ is among $-q^{d-1}, -q^{d-3},\ldots, -q^{1-d}.$ Then the twins of $(a,b,c;q)$ are
$$
(a,b,c;q),\qquad \quad (a^{-1},b^{-1},c^{-1};q^{-1}).
$$
\end{enumerate}
\end{thm}
\begin{proof}
Let $(x,y,z;t)\in T$-$QRAC_{red}$ which is a twin of $(a,b,c;q).$  Let $\Psi$ and $\Psi'$ denote Leonard triple systems that correspond to $(a,b,c;q)$ and $(x,y,z;t)$ via $\pi,$ respectively. Applying Lemma~\ref{lem6.1} to $\Psi$ and $\Psi'$ we find
\begin{align*}
aq^{2i-d}+a^{-1}q^{d-2i}&=xt^{2i-d}+x^{-1}t^{d-2i}\qquad \qquad \hbox{$(0\leq i\leq d),$}\\
b\hspace{0.45mm}q^{2i-d}+b^{-1}q^{d-2i}&=yt^{2i-d}+y^{-1}t^{d-2i}\qquad \qquad \hbox{$(0\leq i\leq d),$}\\
c\hspace{0.45mm}q^{2i-d}+c^{-1}q^{d-2i}&=zt^{2i-d}+z^{-1}t^{d-2i}\qquad \qquad \hbox{$(0\leq i\leq d).$}
\end{align*}
Solving the above three equations for $x,y,z,t$ we find $(x,y,z;t)$ is one of the sequences shown in (\ref{e7.1}). If none of $abc,$ $a^{-1}bc,$ $ab^{-1}c,$ $abc^{-1}$ is among $-q^{d-1}, -q^{d-3},\ldots, -q^{1-d},$ then each sequence shown in (\ref{e7.1}) is in $T$-$QRAC_{red}.$ If some of $abc,$ $a^{-1}bc,$ $ab^{-1}c,$ $abc^{-1}$ is among $-q^{d-1}, -q^{d-3},\ldots, -q^{1-d},$ then among the sequences shown in (\ref{e7.1}) only $(a,b,c;q)$ and $(a^{-1},b^{-1},c^{-1};q^{-1})$ are in $T$-$QRAC_{red}.$ The result follows.
\end{proof}

\section{A $(\mathbb{Z}_2)^3\rtimes S_3$ action on $T$-$QRAC_{red}$}
In Theorem~\ref{thm6.2} we gave a bijection $\pi$ from $T$-$QRAC_{red}$ to $T$-$QRAC.$ In Lemma~\ref{lem1.17} we gave an action of $(\mathbb{Z}_2)^3\rtimes S_3$ on $T$-$QRAC$ as a group of automorphism. Via $\pi^{-1}$ this action induces a $(\mathbb{Z}_2)^3\rtimes S_3$ action on $T$-$QRAC_{red}.$ In this section we describe the resulting $(\mathbb{Z}_2)^3\rtimes S_3$ action on $T$-$QRAC_{red}.$ As we will see, it extends the $D_4$ action on $T$-$QRAC_{red}$ that we obtained in Lemma~\ref{lem1.18}.

\begin{lem}\label{lem1.13}
There exists a unique $(\mathbb{Z}_2)^3\rtimes S_3$ action on
$T$-$QRAC_{red}$ such that
\begin{align}
(a,b,c;q)^*&=(b^{-1},a^{-1},c^{-1};q^{-1}),   \label{e5.2}\\
(a,b,c;q)^\varepsilon&=(c^{-1},b^{-1},a^{-1};q^{-1}),   \label{e5.3}\\
(a,b,c;q)^\downharpoonright&=(a,b,c^{-1};q),   \label{e5.4}\\
(a,b,c;q)^\downarrow&=(a,b^{-1},c;q),    \label{e5.5}\\
(a,b,c;q)^\Downarrow&=(a^{-1},b,c;q)    \label{e5.6}
\end{align}
for all $(a,b,c;q)\in T$-$QRAC_{red}.$
\end{lem}
\begin{proof}
For all $(a,b,c;q)\in T$-$QRAC_{red}$ the sequence on the right in (\ref{e5.2})--(\ref{e5.6}) is contained in $T$-$QRAC_{red}.$ Define maps $*,\varepsilon,\downharpoonright,\downarrow,\Downarrow$ from $T$-$QRAC_{red}$ to $T$-$QRAC_{red}$ such that (\ref{e5.2})--(\ref{e5.6}) hold for all $(a,b,c;q) \in T$-$QRAC_{red}.$
One checks that $*,\varepsilon,\downharpoonright,\downarrow,\Downarrow$ satisfy the relations (\ref{r2.1})--(\ref{r2.4}). Therefore the desired $(\mathbb{Z}_2)^3\rtimes S_3$ action exists. This $(\mathbb{Z}_2)^3\rtimes S_3$ action is unique since $*,\varepsilon,\downharpoonright,\downarrow,\Downarrow$ generate $(\mathbb{Z}_2)^3\rtimes S_3.$
\end{proof}

\begin{lem}\label{lem4.1}
For all $g\in (\mathbb{Z}_2)^3\rtimes S_3$ the following diagram commutes.

\unitlength=1mm
\begin{picture}(120,35)
\multiput(74,6)(0,20){2}{\vector(1, 0){20}}

\put(53,25){{\small $T$-$QRAC_{red}$}}
\put(96,25){{\small $T$-$QRAC$}}
\put(53,5){{\small $T$-$QRAC_{red}$}}
\put(96,5){{\small $T$-$QRAC$}}

\multiput(61,23.5)(40.5,0){2}{\vector(0, -1){15}}
\put(81,28){{\small $\pi$}}
\put(81,2){{\small $\pi$}}
\put(57.5,16){{\small$g$}}
\put(103.5,16){{\small $g$}}
\end{picture}
\end{lem}
\begin{proof}
Without loss we may assume that $g$ is one of $*,\varepsilon,\downharpoonright,\downarrow,\Downarrow.$ Fix $(a,b,c;q)\in T$-$QRAC_{red},$ and let $\Psi$ denote the Leonard triple system that corresponds to $(a,b,c;q)$ via $\pi.$ It is routine to check that $\Psi^g$ and $(a,b,c;q)^g$ correspond via $\pi$ according to Definition~\ref{defn4.1}. The result follows.
\end{proof}

In Lemma~\ref{lem1.18} we gave an action of $D_4$ on $T$-$QRAC_{red}.$ In Lemma~\ref{lem1.13} we gave an action of $(\mathbb{Z}_2)^3\rtimes S_3$ on $T$-$QRAC_{red}.$
Comparing (\ref{e4.6})--(\ref{e4.8}) and (\ref{e5.2}), (\ref{e5.5}), (\ref{e5.6}) we see that the action of
$(\mathbb{Z}_2)^3\rtimes S_3$ on $T$-$QRAC_{red}$ extends the action of $D_4$ on $T$-$QRAC_{red}.$

\section{Classification of Leonard triples of QRacah type}
The goal of this section is to classify up to isomorphism the Leonard triples of QRacah type. Recall the $(\mathbb{Z}_2)^3\rtimes S_3$ action on $T$-$QRAC_{red}$ from Lemma~\ref{lem1.13}.

\begin{defn}\label{defn4.4}
Let $T$-$QRAC_{red}/(\mathbb{Z}_2)^3$ denote the set of all $(\mathbb{Z}_2)^3$-orbits on $T$-$QRAC_{red}.$
\end{defn}

\begin{lem}\label{lem4.5}
Let $(a,b,c;q)$ denote an element of $T$-$QRAC_{red}.$ The $(\mathbb{Z}_2)^3$-orbit containing $(a,b,c;q)$ consists of the eight elements
\begin{align}\label{e8.2}
\begin{array}{llll}
(a, b, c;q), \qquad &(a^{-1}, b, c;q), \qquad &(a, b^{-1}, c;q), \qquad &(a, b, c^{-1};q),\\
(a^{-1}, b^{-1}, c^{-1};q), \qquad \qquad &(a, b^{-1}, c^{-1};q), \qquad \qquad &(a^{-1}, b, c^{-1};q), \qquad \qquad &(a^{-1}, b^{-1}, c;q).
\end{array}
\end{align}
\end{lem}
\begin{proof}
Using (\ref{e5.4})--(\ref{e5.6}) we find that the $(\mathbb{Z}_2)^3$-orbit containing $(a,b,c;q)$ consists of the sequences shown in (\ref{e8.2}). By (T-RQRAC3) none of $a^2, b^2, c^2$ is equal to $1,$ so the sequences shown in (\ref{e8.2}) are mutually distinct.
\end{proof}


\begin{defn}\label{defn4.2}
We define a map $\tau$ from $T$-$QRAC_{red}/(\mathbb{Z}_2)^3$ to the set of all isomorphism classes of Leonard triples over $\mathbb{K}$ that have diameter $d$ and QRacah type. Let $(a,b,c;q)\in T$-$QRAC_{red}.$ Let $(A,A^*,A^\varepsilon)$ denote the Leonard triple over $\mathbb{K}$ that corresponds to $(a,b,c;q)$ via $\pi.$ The map $\tau$ sends the $(\mathbb{Z}_2)^3$-orbit containing $(a,b,c;q)$ to the isomorphism class of $(A,A^*,A^\varepsilon)$ in the set of all Leonard triples. By Lemma~\ref{lem4.1} the map $\tau$ is well-defined.
\end{defn}

\begin{thm}\label{cor4.1}
The map $\tau$ from Definition~\ref{defn4.2} is a bijection.
\end{thm}
\begin{proof}
We first show that $\tau$ is surjective. Let $(A,A^*,A^\varepsilon)$ denote a Leonard triple on $V$ that has QRacah type. By Theorem~\ref{thm6.2} the map $\pi$ is surjective, so there exists an element $(a,b,c;q)$ of $T$-$QRAC_{red}$ that corresponds to $(A,A^*,A^\varepsilon)$ via $\pi.$ Therefore $\tau$ is surjective. We now show that $\tau$ is injective. Suppose we are given $(a,b,c,q)$ and $(a',b',c';q')$ in $T$-$QRAC_{red}$ that correspond to the same Leonard triple $(A,A^*,A^\varepsilon)$ via $\pi.$ We show that $(a,b,c,q)$ and $(a',b',c';q')$ are in the same $(\mathbb{Z}_2)^3$-orbit. Let $\Psi,$ $\Psi'$ denote the Leonard triple systems which are associated with $(A,A^*,A^\varepsilon)$ and correspond to $(a,b,c,q),$ $(a',b',c';q')$ via $\pi,$ respectively. By Definition~\ref{defn2.5} there exists an element $g$ in $(\mathbb{Z}_2)^3$ such that $\Psi'=\Psi^g.$ By this and by Lemma~\ref{lem4.1} the Leonard triple system $\Psi'$ and $(a,b,c;q)^g$ correspond via $\pi.$ Since $\pi$ is injective we see that $(a',b',c';q')=(a,b,c;q)^g.$ We have shown that $\tau$ is injective. The result follows.
\end{proof}

\begin{defn}\label{defn6.1}
For $(a,b,c;q)\in T$-$QRAC_{red}$ define
$$
\widehat{(a,b,c;q)}=(a+a^{-1},b+b^{-1},c+c^{-1};q).
$$
Let $\widehat{T\hbox{-}QRAC_{red}}$ denote the set of sequences $\widehat{(a,b,c;q)}$ where
$(a,b,c;q) \in T$-$QRAC_{red}.$
\end{defn}

\begin{defn}\label{defn6.2}
Observe that the map $(a,b,c;q) \mapsto \widehat{(a,b,c;q)}$ induces a map $T$-$QRAC_{red}/(\mathbb{Z}_2)^3 \to \widehat{T\mbox{-}QRAC_{red}}$ which we denote by $\xi.$
\end{defn}

\begin{thm}\label{thm8.1}
The map $\xi$ from Definition~\ref{defn6.2} is a bijection.
\end{thm}
\begin{proof}
By construction $\xi$ is surjective. To show that $\xi$ is injective, fix $(a,b,c;q)\in T$-$QRAC_{red}$ and let $(x,y,z;t)\in T$-$QRAC_{red}$ that satisfies
\begin{align}\label{e8.3}
\widehat{(x,y,z;t)}=\widehat{(a,b,c;q)}.
\end{align}
Solving (\ref{e8.3}) for $x,y,z,t$ we find that $(x,y,z;t)$ is one of the sequences shown in (\ref{e8.2}). Now the map $\xi$ is injective in view of Lemma~\ref{lem4.5}. The result follows.
\end{proof}

Combining Theorem~\ref{cor4.1} and Theorem~\ref{thm8.1} we find that the following three sets are in bijection:
\begin{enumerate}
\item[$\bullet$] The set of isomorphism classes of Leonard triples over $\mathbb{K}$ that have diameter $d$ and QRacah type.

\item[$\bullet$] The set $T$-$QRAC_{red}/(\mathbb{Z}_2)^3.$

\item[$\bullet$] The set $\widehat{T\mbox{-}QRAC_{red}}.$
\end{enumerate}

\section{Comments}

Let $(A,A^*, A^\varepsilon)$ denote a Leonard triple of QRacah type. We saw in Theorem~\ref{thm2.3} that $A,A^*, A^\varepsilon$ satisfy the $\mathbb{Z}_3$-symmetric Askey-Wilson relations. We now mention some other relations that are satisfied by $A,A^*, A^\varepsilon.$

\begin{prop}\label{prop9.1}
Let $(a,b,c;q)\in T$-$QRAC_{red}.$ Let $(A,A^*,A^\varepsilon)$
denote the Leonard triple over $\mathbb{K}$ that corresponds to $(a,b,c;q)$ via
$\pi.$ Then $\psi I$ is equal to each of
\begin{align*}
\begin{split}
&qAA^*A^\varepsilon+q^{2}A^2+q^{-2}A^{*2}+q^2A^{\varepsilon2}-q\alpha A-q^{-1}\alpha^*A^*-q\alpha^\varepsilon A^\varepsilon,\\
&qA^\varepsilon AA^*+q^{2}A^{\varepsilon2}+q^{-2}A^2+q^2A^{*2}-q \alpha^\varepsilon A^\varepsilon-q^{-1}\alpha A-q \alpha^* A^*,\\
&qA^*A^\varepsilon A+q^{2}A^{*2}+q^{-2}A^{\varepsilon2}+q^2A^2-q\alpha^* A^*-q^{-1}\alpha^\varepsilon A^\varepsilon-q\alpha A,\\
&q^{-1}A^*AA^\varepsilon+q^{-2}A^{*2}+q^{2}A^2+q^{-2}A^{\varepsilon2}-q^{-1}\alpha^* A^*-q\alpha A-q^{-1}\alpha^\varepsilon A^\varepsilon,\\
&q^{-1}A^\varepsilon A^*A+q^{-2}A^{\varepsilon2}+q^{2}A^{*2}+q^{-2}A^2-q^{-1}\alpha^\varepsilon A^\varepsilon-q\alpha^* A^*-q^{-1}\alpha A,\\
&q^{-1}AA^\varepsilon A^*+q^{-2}A^2+q^{2}A^{\varepsilon2}+q^{-2}A^{*2}-q^{-1}\alpha A-q\alpha^\varepsilon A^\varepsilon-q^{-1}\alpha^* A^*,
\end{split}
\end{align*}
where
\begin{align*}
\psi&=(q+q^{-1})^2-(q^{d+1}+q^{-d-1})^2-(a+a^{-1})^2-(b+b^{-1})^2-(c+c^{-1})^2\\
&\qquad -(a+a^{-1})(b+b^{-1})(c+c^{-1})(q^{d+1}+q^{-d-1})
\end{align*}
and
\begin{align*}
\alpha&=(b+b^{-1})(c+c^{-1})+(a+a^{-1})(q^{d+1}+q^{-d-1}),\\
\alpha^*&=(c+c^{-1})(a+a^{-1})+(b+b^{-1})(q^{d+1}+q^{-d-1}),\\
\alpha^\varepsilon&=(a+a^{-1})(b+b^{-1})+(c+c^{-1})(q^{d+1}+q^{-d-1}).
\end{align*}
\end{prop}
\begin{proof}
Use the matrix forms for $A,A^*,A^\varepsilon$ displayed below Definition~\ref{defn4.3}.
\end{proof}

\section{Acknowledgements}
This work was realized while the author visited the Department of Mathematics, University of
Wisconsin-Madison during January 17th and December 14th, 2010. The author would like to thank Prof. Paul Terwilliger for his many valuable ideas and suggestions.
Most of the calculations were performed using Maple software. This research was supported by the NSC grant 98-2917-I-009-113 of Taiwan.

\end{document}